\documentclass[10pt]{amsart}

\usepackage{latexsym, amsmath,amssymb}
\usepackage{color}

\renewcommand{\epsilon}{\varepsilon}
\newcommand{\newsection}[1]
{\subsection{#1}\setcounter{theorem}{0} \setcounter{equation}{0}
\par\noindent}

\newtheorem{theorem}{Theorem}

\newtheorem{lemma}[theorem]{Lemma}
\newtheorem{corr}[theorem]{Corollary}
\newtheorem{prop}[theorem]{Proposition}
\newtheorem{deff}[theorem]{Definition}

\newcommand{\bth}{\begin{theorem}}
\newcommand{\ble}{\begin{lemma}}
\newcommand{\bcor}{\begin{corr}}

\newcommand{\bdeff}{\begin{deff}}

\newcommand{\bprop}{\begin{proposition}}
\newcommand{\ele}{\end{lemma}}
\newcommand{\ecor}{\end{corr}}
\newcommand{\edeff}{\end{deff}}

\newcommand{\eprop}{\end{proposition}}

\newcommand{\norm}[2]{\left\| #1 \right\|_{#2}}
\renewcommand{\Pi}{\varPi}

\renewcommand{\epsilon}{\varepsilon}

\newcommand{\K}{{\mathcal K}}
\newcommand{\R}{{\mathbb R}}

\newcommand{\la}{{\langle}}
\newcommand{\ra}{{\rangle}}
\newcommand{\cd}{{\,\cdot\,}}
\newcommand{\bdy}{{\partial\K}}
\newcommand{\ext}{{\R^4\backslash\K}}
\newcommand{\ssum}[2]{\sum_{\substack{ #1 \\ #2}}}
\renewcommand{\S}{{\mathbb{S}}}

\begin{document}

\title[Almost global existence of quasilinear wave equations]
{
Almost global existence for 4-dimensional quasilinear wave equations in exterior domains
}

\author{John Helms}
\address{Department of Mathematics, University of California, 
  Santa Barbara, CA 93106-3080}
\email{johnhelms@math.ucsb.edu}

\author{Jason Metcalfe}
\address{Department of Mathematics, University of North Carolina,
  Chapel Hill, NC  27599-3250}
\email{metcalfe@email.unc.edu}

\thanks{The second author was supported in part by NSF grant DMS-1054289.}

\begin{abstract}
This article focuses on almost global existence for quasilinear wave equations
with small initial data in 4-dimensional exterior domains. The nonlinearity is
allowed to depend on the solution at the quadratic level as well as its
first and second derivatives.  For this problem in the boundaryless
setting, H\"{o}rmander proved that the lifespan
is bounded below by $\exp(c/\epsilon)$ where $\epsilon>0$ denotes the
size of the Cauchy data.  Later Du, the second author,
Sogge, and Zhou showed that this inequality also holds for 
star-shaped obstacles. Following up on the authors' work in the
3-dimensional case, we weaken the hypothesis on the geometry and only require
that the obstacle allow for a sufficiently rapid decay of local energy
for the linear homogeneous wave equation.  The key innovation of this paper
is the use of the boundary term estimates of the second author and Sogge
in conjunction with a variant of an estimate of Klainerman and
Sideris, which will be obtained via a Sobolev inequality of Du and Zhou.
\end{abstract}

\maketitle

\newsection{Introduction}
In this article, we establish a lower bound of $\exp(c / \epsilon)$ on 
the lifespan of small-data solutions to quasilinear wave equations in 
4-dimensional exterior domains with Dirichlet boundary conditions.  Here $\epsilon$ 
denotes the size of the Cauchy data in a suitably chosen Sobolev norm.
 As the lifespan grows exponentially as the size of the initial data
 shrinks, the solution is said to exist almost globally.  The nonlinearities that
we are considering may depend on the solution $u$ in addition to its first and second
derivatives at all levels. 
The lifespan bound established in this article
was first proved in \cite{Hormander} for boundaryless wave equations. A 
relatively recent paper \cite{DMSZ} established 
the same lifespan bound for the exterior of star-shaped domains. This article relaxes the geometric assumptions to
allow for domains in which there is a sufficiently rapid decay of local energy with a possible loss
in $L^2$ regularity.

We now introduce the problem at hand.  Let $\K\subset \R^4$
be a bounded domain with smooth boundary.  Note that we shall not
assume that $\K$ is connected.  We then examine the following
quasilinear wave equation exterior to $\K$
\begin{equation}
  \label{main}
  \begin{cases}
    \Box u(t,x)=Q(u,u',u''),\quad (t,x)\in \R \times\ext,\\
    u(t,\cd)|_{\bdy}=0,\\
    u(0,\cd)=f,\quad \partial_t u(0,\cd)=g.
  \end{cases}
\end{equation}
Here $\Box=\partial_t^2-\Delta$ is the d'Alembertian, and
$u'=\partial u = (\partial_t u, \nabla_x u)$ denotes the space-time gradient.  
Throughout this paper, $u$ will refer to the solution to \eqref{main} with initial data $f,g$.
The nonlinear term $Q$ vanishes to second-order at the origin 
and is linear in $u''$.  Due to the fact that the wave equation is invariant under 
scaling and translations, we shall take $0 \in \K \subset \{|x|<1\}$,
without a loss of generality, throughout the paper.  While we shall state the lifespan bound for 
the scalar equation \eqref{main}, the methods we shall use can be easily adapted 
to prove almost global existence for systems of wave equations, even with multiple wave speeds.

The nonlinearity $Q$ can be expanded as
\[
Q(u,u',u'')=A(u,u') +
B^{\alpha\beta}(u,u')\partial_\alpha\partial_\beta u,
\]
where $A(u,u')$ vanishes to second order at the origin and
$B^{\alpha\beta}$ are functions which are symmetric in $\alpha,\beta$
and vanish to first
order at $(0,0)$.  Here we are using the summation convention where
repeated indices are implicitly summed from $0$ to $4$, $x_0=t$,
$\partial_0=\partial_t$, and $\partial_\alpha = \partial_{x_\alpha}$ 
for $1 \leq \alpha \leq 4$. We will also frequently use multi-index notation,
setting e.g. $\partial^\mu = \partial^{\mu_0}_0 \partial^{\mu_1}_1 
\cdots \partial^{\mu_4}_4$ where $\mu = (\mu_0, \ldots , \mu_4)$.

Since we are working with small Cauchy data, the arguments used to control the quadratic terms can be easily adapted to handle the higher order terms in the Taylor expansion of $Q$ about $(0,0,0)$. 
Thus, we shall truncate $Q$ at the quadratic level. We may write
\[Q(u,u',u'') = A(u,u') + b^{\alpha\beta}
u \partial_\alpha\partial_\beta u +
b^{\alpha\beta}_{\gamma} \partial_\gamma
u \partial_\alpha\partial_\beta u,\]
where $b^{\alpha\beta}$ and $b^{\alpha\beta}_\gamma$ are real constants
which are symmetric in $\alpha, \beta$ and $A(u,u')$ is a quadratic form.

Solving \eqref{main} requires the Cauchy data to satisfy compatibility
conditions.  Formally, for a solution $u\in H^m$, we write
$\partial_t^k
u(0,\cd)=\psi_k(J_kf, J_{k-1}g)$, $0\le k\le m$
where
$J_ku=\{\partial_x^\mu u\,:\,0\le |\mu|\le k\}$.
For $(f,g)\in H^m\times H^{m-1}$, the compatibility condition
requires that the compatibility functions, $\psi_k$, vanish on $\bdy$
for all $0\le k\le m-1$.  For smooth data, we require the
compatibility conditions to hold for all $m$.  See \cite{KSS2} for a
more detailed description of the compatibility conditions.


Our only geometric assumption on $\K$ requires the local energy
for solutions to the linear homogeneous wave equation with compactly
supported data to decay
at a sufficiently rapid, fixed, algebraic rate.  We allow for a loss of $D$ derivatives.  Our proof permits
a loss of any fixed number of derivatives, though we do require that
the local energy decay at a rate faster than $t^{-2}$.
More specifically, we assume that
there are fixed constants $\sigma > 0$ and 
$D \geq 0$ such that if $\Box u = 0$ and if the Cauchy
data $u(0,\cdot), \partial_t u(0,\cdot)$ are supported on the set $\{ |x| < 10 \}$, then the following inequality holds
\begin{equation} \label{local energy}
 \norm{u'(t,\cd)}{L^2( \{ x \in \ext : |x| < 10 \})} \lesssim \left< t
 \right>^{-2-\sigma} 
\sum_{|\mu| \leq D} \norm{\partial^\mu u'(0,\cd)}{2} .
\end{equation}
Here $\left< t \right> = (1 + t^2)^{1/2}$. The notation $A\lesssim B$ indicates that there is a
positive unspecified constant $C$, which may change from line to line,
so that $A\le CB$.  Moreover, this $C$ will implicitly be independent
of any important parameters in our problem.

Local energy decay estimates such as \eqref{local energy} have an extensive history. 
We shall only briefly describe some past results that directly relate to the problem at hand. 
For nontrapping obstacles, it was initially shown in $n=3$ that no
loss ($D=0$) in the right 
hand side is necessary in order to obtain exponential decay.  See,
e.g., \cite{MRS}.  Moreover, \cite{Ralston} showed that a loss,
e.g. $D\neq 0$, must
occur in \eqref{local energy} when there is trapping.  In the other
direction, \cite{Ikawa1, Ikawa2} first established examples in 3
spatial dimensions of
geometries with trapping for which \eqref{local energy} holds, and
there have been many subsequent works in this direction in odd dimensions.
In even dimensions $n$, it was shown in 
\cite{Ral} that the local energy decays at a rate of 
$O(t^{-(n-1)})$ when there is no trapping ($D=0$). See also \cite{Mel} and 
\cite{Strauss}.  Thus, the assumption \eqref{local energy} is a weaker assumption than was
made in \cite{DMSZ}.





We may now state our main theorem, which shows that for Cauchy data of
size $\varepsilon$ solutions to \eqref{main} must exist up to $T_\varepsilon = \exp ( c / \epsilon)$ for some small constant
$c$.
\begin{theorem}\label{thm1}
 Let $\K$ be a smooth, bounded domain for which \eqref{local energy} holds, and let $Q$ be as
 above.  Suppose that the Cauchy
 data $f,g\in C^\infty(\ext)$ are compactly supported and satisfy the compatibility conditions to
 infinite order.  Then there exist constants $N$ and $c$ so that if
 $\varepsilon$ is sufficiently small and
  \begin{equation}\label{data}
   \sum_{|\mu|\le N} \|\partial_x^\mu f\|_{2} + \sum_{|\mu|\le
     N-1}\|\partial_x^\mu g\|_2 \le \varepsilon,
  \end{equation}
then \eqref{main} has a unique solution $u\in
C^\infty([0,T_\varepsilon]\times \ext)$ where
\begin{equation}
  \label{lifespan}
T_\varepsilon =  \exp ( c / \epsilon ).
\end{equation}
\end{theorem}


We are assuming here that the Cauchy
data are compactly supported.  It is likely that it would suffice to
take the data to be small in certain weighted Sobolev norms.

The lifespan in Theorem \ref{thm1} was first proved in
\cite{Hormander} for boundaryless wave equations in 4 dimensions.
However, \cite{LX} and \cite{LS2} demonstrate
that this bound for the boundaryless case can be improved to $T_\epsilon \gtrsim \exp(c / \epsilon^2)$.
Thus, it may be possible that the lifespan bound presented in this paper can be improved.
In the other direction, it was shown in \cite{Sideris} and later in \cite{Zhou2} and \cite{Yordanov} that solutions to \eqref{main} must blow up in finite time. 
See also \cite{ZhouHan} for related blow up results for semilinear wave equations.
When $Q$ only depends on $u'$ and $u''$, it is well-known that
solutions corresponding to sufficiently small data exist globally.
See, e.g., \cite{Klainerman2}, \cite{Htext}, \cite{So2}, \cite{H3}.
The dependence of $Q$ on the solution
$u$ itself instead of only its first and second order derivatives hinders
many of the energy methods that are commonly employed in proving lifespan bounds.

For exterior domains, \cite{MS2} proved small data global existence
under similar conditions on the geometry when $n\ge 4$ and $Q$ is
independent of $u$.
The first paper establishing lifespan bounds for exterior domain problems in case 
$Q$ depends on $u$ at the quadratic level is \cite{DZ}.  They
established an analog of the lifespan bound $T_\epsilon \gtrsim \epsilon^{-2}$ of \cite{Lindblad} for 
3 dimensional wave equations outside of a star-shaped obstacle.  A subsequent paper 
\cite{DMSZ} proved an exterior domain analog of the almost global
existence theorem of
\cite{Hormander} for star-shaped obstacles in 4 spatial dimensions.
It is precisely this geometric condition that we are seeking to relax.
Moreover \cite{MS5} shows global existence exterior to star-shaped
obstacles provided $Q_{uu}(0,0,0)=0$, which is an exterior domain
analog of another result of \cite{Hormander}.

The previous paper of the authors \cite{HM}
extended the result of \cite{DZ} to exterior domains that contain
trapped rays.  This paper is a follow-up
to our previous work and proves the lifespan bound of \cite{Hormander} and \cite{DMSZ}
under weaker geometric assumptions.


The proof shall utilize the method of invariant
vector fields \cite{Klainerman}, which was adapted to the exterior domain
setting in \cite{KSS, KSS3} and \cite{MS3,MS2}. Although this paper primarily concerns solutions in 4 spatial dimensions, many of the estimates we shall state also hold in other dimensions. In $\R \times \R^n$, we set 
\begin{equation} \label{invariant vfs}
Z=\{\partial_\alpha, \Omega_{ij}=x_i\partial_j-x_j\partial_i\,:\,
0\le \alpha\le n,\, 1\le i<j\le n \},
\end{equation}
and $L=t\partial_t + r\partial_r$ where $r=|x|$.
The $\Omega_{ij}$ are the generators of spatial rotations in $\R^n$, and 
$L$ is the scaling vector field.  We shall frequently make
use of the fact that
\[[\Box,Z]=0,\quad [\Box, L] = 2\Box.\]
These vector fields are regarded as invariant since they share the property that
if $\square u =0$, then $\square Zu = 0$ and $\square Lu = 0$. 
In earlier works such as \cite{Klainerman2}, Lorentz boosts 
$\Omega_{0j} = t \partial_j + x_j \partial_t$ were also included in 
the collection \eqref{invariant vfs} when dealing 
with single-speed wave equations in the case that there is no boundary.
We omit them since they seem to be compatible with neither 
Dirichlet boundary value problems nor with systems of wave equations with
multiple wave speeds. In particular, boosts are not favorable for these sorts of
problems since they have an unbounded normal component
on $\bdy$ as well as an associated wave speed: 
$[ ( \partial_t^2 - c^2 \Delta ) , \Omega_{0j} ] u \neq 0$ if $c \neq 1$.
While all members of $Z$ commute nicely with $\square$, only the
generators of 
time translations $\partial_t$ also preserve the Dirichlet boundary conditions. 
However, \cite{KSS} demonstrated how to adapt the methods to the
remaining vector fields in $Z$ using that they
almost preserve the Dirichlet boundary conditions in the sense that 
their normal components are uniformly bounded on
$\bdy$.  In particular, this allows them to be handled using elliptic regularity 
and localized energy estimates.
Although the scaling vector field has coefficients that are unbounded for large $t$,
its normal component is also bounded on $\bdy$.  Thus, while it may be
used, we will need to employ few $L$ relative to
the number of vector fields from $Z$. This method of proof was started
in \cite{KSS3} and was further developed in
\cite{MS3}, \cite{MNS, MNS2} and \cite{HM}. 

The previous result \cite{DMSZ} assumed
that $\K$ is star-shaped in order to prove useful energy and localized energy estimates 
for variable coefficient wave equations.
The techniques used in their paper are reminiscent of \cite{MS}, which reproved the
earlier result of \cite{KSS3} without using the scaling vector field $L$.
It is not clear that one can obtain analogous localized energy estimates under
the weaker geometric assumption \eqref{local energy} that we are using. 
To make up for the lack of such localized estimates when $\K$ is not star-shaped, in \cite{HM}
the authors used H\"{o}rmander's $L^1-L^\infty$ estimate \cite{Htext} that was adapted to the exterior domain 
setting in \cite{KSS3}. A higher dimensional analogue of this estimate
was proven in \cite{HormanderL1Linfty}, but that estimate involves
Lorentz boosts and does
not seem easily applicable to the problem at hand since, at this point
in time, the authors are not aware of a proof that eliminates the boosts. 

The main innovation of this paper is to combine the methods of \cite{DZ}, \cite{MS3} and \cite{KlainermanSideris} to prove almost global existence. This includes using the estimates of
\cite{DZ} and \cite{DMSZ}
to obtain a useful pointwise estimate for $u$ rather than $u^\prime$ (cf. \cite[Lemma 4.2]{MS2}), whose role in our proof will be analogous to the role of H\"{o}rmander's $L^1 - L^\infty$ estimate \cite{Htext} in previous results such as \cite{KSS3}, \cite{MS3}, \cite{DZ}, and \cite{HM}. Just
as with H\"{o}rmander's estimate, our new estimate also necessitates the use of the scaling 
vector field in our proof.
In previous papers, such as \cite{DZ} and \cite{KSS3}, the star-shaped assumption
allowed one to use methods similar to those of Morawetz \cite{Morawetz}
to show that the worst boundary term resulting from $L$ has a 
beneficial sign and, therefore, can be ignored.
We shall control norms that include $L$ in a manner similar to \cite{MS3,MS2}
by using boundary term estimates that use our local energy decay assumption
\eqref{local energy}. 

The remainder of the article is organized as follows.  In Section 2,
we state our main energy and localized energy estimates. For the most part, these
consist of energy and localized energy estimates combined with the Sobolev estimate used in \cite{DZ} and \cite{DMSZ} 
as well as the main estimates of \cite{MS3,MS2}, 
which enable the use of the scaling vector field 
in exterior domains where $\K$ is not star-shaped. 
In Section 3, we state our main pointwise estimates. These include a well-known application
of the Sobolev embedding theorem on annuli (see \cite{Klainerman}). In this section we also combine 
the estimate of \cite{DZ} with an estimate that is similar to those 
of \cite{KlainermanSideris} 
to establish our main pointwise dispersive estimate. This is 
reminiscent of the estimates appearing in e.g. \cite{Klainerman}, \cite{KlainermanSideris}, 
\cite{Sideris2}, \cite{SiderisTu}, \cite{H3} and \cite{HY}.
In Section 4, we prove almost
global existence of small-data solutions to \eqref{main} as stated in Theorem \ref{thm1}.

\noindent \textbf{Acknowledgements:}
The first author would like to thank Thomas Sideris for the helpful conversations during the process of writing the introduction to this paper.

\newsection{$L^2$ estimates}

In this section, we state the main $L^2$ estimates that we shall need in
our iteration argument. These estimates largely consist of well-known energy and localized
energy estimates of the solution when vector fields from the collection $\{L,Z\}$ are being applied.
The most basic energy and localized energy estimate 
for wave equations on $\R \times \R^4$ is the following.
\begin{multline} \label{energy}
\sup_{t \in [0,T]} \norm{v'(t,\cdot)}{2} 
+ \sup_{R \geq 1} R^{-1/2} \norm{v'}{L^2_{t,x}([0,T] \times \{ |x| < R \} )}
\lesssim \norm{v'(0,\cdot)}{2} \\
+ \inf_{\square v = f + g} \left( \int^T_0 \norm{f(s,\cdot)}{2} \; ds 
+ \sum_{j \geq 0} \norm{\left< x \right>^{1/2} g}{L^2_{t,x}([0,T] \times \{ \left< x \right> \approx 2^j \} )} \right).
\end{multline}
Most of the $L^2_x$ and weighted $L^2_{t,x}$ estimates for $v'$ presented in this paper will be variants of the above estimate.
A version of \eqref{energy} was first proved in \cite{Morawetz2} and
was applied in many subsequent papers.  See, e.g., \cite{MS, MS4}, \cite{MT2} for 
versions of the estimate that are more reminiscent of \eqref{energy}
and for a more complete history.  In \cite{KSS,KSS3}, such localized energy estimates were applied to 
prove long time existence in the exterior domain setting.
Since then nearly every long time existence result for obstacle problems such as \eqref{main} 
has employed localized energy estimates similar to \eqref{energy}.  An
important consequence of \eqref{energy} that we will apply throughout this
paper is the following.
\begin{lemma} \label{main energy}
Let $v \in C^\infty (\R \times \R^4)$, and assume that $v$ vanishes for large $|x|$ for every $t$. Then for $T \geq 1$, we have
\begin{multline} \label{main energy ineq}
\norm{ \left< x \right>^{-3/4} v' }{L^2_{t,x}([0,T] \times \R^4)} + \log(2 + T)^{-1/2} \norm{ \left< x \right>^{-1/2} v' }{L^2_{t,x}([0,T] \times \R^4)} \\
\lesssim \norm{v'(0,\cdot)}{2} 
+ \inf_{\square v = f + g} \left( \int^T_0 \norm{f(s,\cdot)}{2} \; ds 
+ \sum_{j \geq 0} \norm{\left< x \right>^{1/2} g}{L^2_{t,x}([0,T] \times \{ \left< x \right> \approx 2^j \} )} \right).
\end{multline}
\end{lemma}
The proof of this lemma is fairly straightforward. In the case where
the spatial norms in the left hand side of \eqref{main energy ineq} are taken over $\{ |x| > T \}$, the left hand side of \eqref{main energy ineq} is bounded by the first term in the left hand side of \eqref{energy}. For $|x| < T$, we decompose
dyadically in $\left< x \right>$ and apply \eqref{energy} to each
term.  See, e.g., \cite{KSS}, \cite{M2}, \cite{MS, MS4} for more
detailed arguments of this sort.

\subsubsection{Estimates for $\|v\|_{L^2_x}$ on $\R \times\R^4$}\label{bdylessSection}
In this section, we seek to establish control of the $L^2$ norm of the
solution without the space-time gradient.  It will suffice to prove
our estimates for boundaryless wave equations.  In the sequel, we
shall cutoff away from the obstacle to reduce to this case.  Near the
boundary, we note that Dirichlet boundary conditions
($v(t,\cd)\bigl|_{\bdy} = 0$) imply the inequality
\begin{equation} \label{local control}
\norm{v(t,\cdot)}{L^2(\{ x \in \ext : |x| < 2\} )} \lesssim \norm{v'(t,\cdot)}{L^2(\{ x \in \ext : |x| < 2\})},\end{equation}
which allows us to control the $L^2$ norm of the solution in terms of the
energy in this remaining region.

We now state the following Sobolev-type 
estimate of \cite{DZ}.  Ignoring the behavior near the origin, the
estimate is roughly the dual to estimates appearing in \cite{Sideris2}.  For detailed proofs the reader should consult \cite{DZ} and \cite{DMSZ}.
\begin{lemma} \label{DZ sobolev estimate}
Let $n \geq 3$ and $h \in C^\infty_0 (\R^n)$. Then it follows that
\begin{equation} \label{DZ estimate ineq}
\norm{h}{\dot{H}^{-1}} \lesssim \norm{h}{L^{2n/(n+2)}(|x| < 1)} + \norm{|x|^{-(n-2)/2} h}{L^1_r L^2_\omega (|x| > 1)} .
\end{equation}
\end{lemma}
{\noindent The mixed norm appearing in \eqref{DZ estimate ineq} is
  defined as}
\[
\norm{h}{L^p_r L^q_\omega (|x| > 1)} = \left( \int^\infty_1 \left[ \int_{\S^{n-1}} \left| h(r\omega) \right|^q \; d\omega \right]^{p/q} \; r^{n-1}\; dr \right)^{1/p} .
\]

Using this lemma, one obtains the following proposition, which first appeared in \cite{DZ} and \cite{DMSZ}.
\begin{prop} \label{u bound prop}
Let $v \in C^\infty (\R \times \R^4)$, and assume that $v$ vanishes for
large $|x|$ for every $t$. Then for any $T \geq 1$, we have
\begin{multline} \label{u bound}
		\norm{v}{L^\infty_tL^2_x([0,T]\times \R^4)} + \log( 2 + T )^{-1/2} \norm{\la x\ra^{-1/2}
    v}{L^2_{t,x}([0,T]\times \R^4)} \lesssim \norm{v(0,\cd)}{2} +
    \norm{\partial_t v(0,\cd)}{\dot{H}^{-1}} \\
		+ \int_0^T \norm{|x|^{-1} \Box v(s,\cd)}{L^1_rL^2_\omega(|x|>1)}\;ds +
		\int_0^T \norm{\Box v(s,\cd)}{L^{4/3}(|x|<1)}\;ds.
\end{multline}
\end{prop}
To prove Proposition \ref{u bound prop}, one needs only to apply
\eqref{main energy ineq} to a combination of the Riesz transforms of
$u$. After doing so, applying Lemma \ref{DZ sobolev estimate} finishes the proof.

To handle the commutator terms that result from applying a spatial
cutoff to the solution, we argue similarly but instead measure the
forcing term using the last term in the right of \eqref{main energy
  ineq}.  Indeed, by arguing as above, if $v$ solves
\begin{equation}\label{linear v}
\begin{cases}
  \Box v = G,\quad (t,x)\in \R\times \R^4,\\
   v(t,x)=0,\quad t\le 0,
\end{cases}\end{equation}
we get
\begin{multline} \label{inv laplace}
\norm{v}{L^\infty_t L^2_x ([0,T] \times \R^4)} + \log(2 + T )^{-1/2} \norm{\left< x \right>^{-1/2} v}{L^2_{t,x} ([0,T] \times \R^4)} \\
\lesssim \sum_{j=1}^4 \norm{\left< x \right>^{1/2 + \delta}(\Delta^{-1} \partial_j G)}{L^2([0,T] \times \R^4)},
\end{multline}
for any $\delta>0$.
As the kernel of the operator $\Delta^{-1} \partial_j$ is $O(|x -
y|^{-3})$ for $j=1,\dots,4$, it follows from Young's inequality that
\[
\norm{\left< x \right>^{1/2 + \delta} (\Delta^{-1} \partial_j G) (t,\cdot)}{2} \lesssim \norm{G (t,\cdot)}{2}, 
\]
if $ G(t,x) = 0$ for $|x|  > 3$ and $0 < \delta < 1/2$. 
Combining this inequality with \eqref{inv laplace}, we obtain the following estimate, which appeared in \cite{DMSZ}.
\begin{prop}
Let $v \in C^\infty ( \R \times \R^4 )$ be a solution to \eqref{linear
  v} where $G(t,x) = 0$ for $|x| > 3$.  Then it follows that 
\begin{equation} \label{cutoff estimate}
\norm{v}{L^\infty_t L^2_x ([0,T] \times \R^4)} + \log(2 + T )^{-1/2} \norm{\left< x \right>^{-1/2} v}{L^2_{t,x} ([0,T] \times \R^4)} \lesssim \norm{G}{L^2_{t,x} ([0,T] \times \{ |x| < 3 \})} .
\end{equation}
\end{prop}

We shall need to take advantage of better bounds than those provided
by \eqref{u bound}
when the forcing term  is in divergence form. 
The following estimate appeared in \cite{MS5} and was inspired by the
similar estimates of \cite{Hormander} and \cite{Lindblad}.  See
\cite{HM} for a similar application in the 3-dimensional setting.
\begin{prop} \label{divergence form prop}
Let $v \in C^\infty ( \R \times \R^4 )$ solve
\begin{equation}
\left\{
\begin{array}{l}
\square v(t,x) = \sum_{j=0}^4 a_j \partial_j G(t,x) , \quad (t,x) \in \R \times \R^4, \\
v(0,x) = 0, \quad t \leq 0,
\end{array}
\right.
\end{equation}
where $a_j \in \R$. Also suppose that $v(t,x)$ vanishes for large $|x|$ for every $t$. Then it follows that
\begin{multline} \label{divergence form}
\norm{v}{L^\infty_tL^2_x([0,T]\times \R^4)} + \log( 2 + T )^{-1/2} \norm{\la x\ra^{-1/2}
    v}{L^2_{t,x}([0,T]\times \R^4)} \\
    \lesssim \norm{G(0,\cdot)}{\dot{H}^{-1}} + \int^T_0 \norm{G(s,\cdot)}{2} \; ds.
\end{multline}
\end{prop}
Here we need only observe that $v = \sum^4_{j=0} a_j \partial_j v_1 -
a_0 v_2$, where $\square v_1=  G$ with vanishing initial data and
$\square v_2 = 0$ with initial data $v_2(0,x) = 0, \partial_t v_2(0,x)
= G(0,x)$, and apply \eqref{main energy ineq} and \eqref{u bound}.

\subsubsection{Energy estimates on $\R \times \ext$}
We will need pointwise in time $L^2$ energy estimates 
for the proof of Theorem \ref{thm1}. 
Since we are concerned with solutions to quasilinear wave equations, we will need estimates for solutions to
perturbed wave equations.
When the solution satisfies the Dirichlet boundary conditions,
these estimates are well-known for the solution $v$ itself.  Indeed,
as the Dirichlet boundary conditions imply that $\partial_t v(t,x)=0$
for any $x\in \bdy$, such estimates follow from the standard energy
integral calculations that can be found in texts such as \cite{Htext}
and \cite{Sogge}. 

However, when one applies a vector field from our collection $\{ L , Z \}$, 
the Dirichlet boundary conditions are only preserved by time translations $\partial_t$.
Hence, it is more difficult to control the boundary terms that result from integrating by
parts.  A heirarchy of vector fields results.  From an estimate
involving only $\partial_t$ vector fields, using elliptic regularity,
an estimate for any $\partial_{t,x}$ vector field results.  Using that
the coefficients of $Z$ are $O(1)$ on the boundary of the compact
obstacle, the standard argument yields an estimate up to an error term
that can be controlled by localized energy estimates for $\partial^\mu
v$.  Here, however, we must permit more vector fields of the form
$\partial$ rather than $Z$.  Due to the large coefficient, particular care must be paid to
controlling $L$ on the boundary of $\K$.  Here we shall apply a cutoff
to the vector field $L$ so that the Dirichlet boundary conditions are
preserved by this new vector field.  The resulting commutator will
only involve vector fields that are higher in the heirarchy, and the
resulting error term is controlled by pointwise estimation of the
solution.  This yields estimates for $L \partial^\mu v$, and estimates
for $L Z^\mu$ follow, again, via a localized energy estimate.

This adaptation of the vector field method to exterior domains
originates in \cite{KSS, KSS2}.  In \cite{KSS2}, star-shapedness was
used to control the worst boundary term that resulted when applying
the scaling vector field.  The method saw particular further development
in \cite{MS3} where the idea of applying a cutoff to the scaling
vector field was first used to permit more general geometries.  The
estimate for the resulting commutator term of \cite{MS3} in
3-dimensions relied on Huygens' principle, and this method was adapted
in \cite{MS2} to generic dimension $n\ge 4$.

Here we are merely gathering the estimates of \cite{MS2, MS3},
and unless otherwise specified we refer the reader there for detailed proofs.

In particular, we are interested in solutions to 
\begin{equation} \label{perturbed wave}
	\left\{
	\begin{array}{l}
		\square_\gamma v (t,x) = G (t,x) , \quad (t,x) \in \R \times \ext, \\
		v(t,x) = 0 , \quad x \in \partial \K, \\
		v(0,x) = f(x) , \quad \partial_t v(0,x) = g(x),
	\end{array}
	\right.
\end{equation}
where
\[
	\square_\gamma = (\partial_t^2 - \Delta)  + \gamma^{\alpha \beta}(t,x) \partial_\alpha
	\partial_\beta .
\]
The perturbation terms $\gamma^{\alpha \beta}$ satisfy $\gamma^{\alpha \beta} = \gamma^{\beta \alpha}$ as well as
\begin{equation} \label{perturbation small}
	\norm{\gamma^{\alpha \beta} (t,\cdot)}{\infty} 
	\leq \dfrac{\delta}{1+t}, \quad 0 < \delta \ll 1.
\end{equation}
We shall also use the notation
\[
	\norm{\gamma^\prime(t,\cdot)}{\infty} 
	= \sum^4_{\alpha,\beta, \mu =0} \norm{\partial_\mu \gamma^{\alpha \beta}(t,\cdot)}{\infty}
\]
and the perturbation will be chosen so that
\begin{equation}\label{gamma prime small}\|\gamma'(t,\cdot)\|_{L^\infty} \le \frac{\delta}{1+t},\quad
0<\delta\ll 1.
\end{equation}


\par
We set $e_0(v)$ to be the energy form
\[
	e_0 (v) = \left| v^\prime \right|^2 + 2 \gamma^{0
          \alpha} \partial_t v \partial_\alpha v 
	- \gamma^{\alpha \beta} \partial_\alpha v \partial_\beta v.
\]
The first estimate involves the quantity
\[
	E_M(t) = E_M(v)(t) := \int_{\ext} \sum^M_{j=0} e_0(\partial_t^j v)(t,x) \; dx .
\]
Due to the fact that the vector field $\partial_t$ preserves the Dirichlet boundary conditions specified in \eqref{perturbed wave}, standard energy methods yield the following estimate.
\begin{lemma} \label{perturbed energy lemma}
Fix $M = 0 , 1 , 2 , \ldots$ and assume that the perturbation terms $\gamma^{\alpha \beta}$ are as above. Suppose also that $v \in C^\infty (\R \times \ext ) $ solves \eqref{perturbed wave} and vanishes for large $|x|$ for every $t$. Then
\begin{equation} \label{perturbed energy ineq}
\partial_t E_M^{1/2} (t) \lesssim \sum^M_{j=0} \norm{\square_\gamma \partial_t^j v(t,\cdot)}{2} + \norm{\gamma^\prime(t,\cdot)}{\infty} E_M^{1/2} (t) .
\end{equation}
\end{lemma}

In the proof of Theorem \ref{thm1}, we shall frequently make use of the fact that
\[
e_0(v) \approx |v'|^2
\]
provided that \eqref{perturbation small} holds for $\delta$
sufficiently small.  From the class of estimates provided by Lemma
\ref{perturbed energy lemma}, one can then use elliptic regularity to control $L^2$ norms involving $\partial^\mu_{t,x} v'$. This
approached was used in \cite{KSS,KSS3} and \cite{MS3} amongst others.
See also, e.g.,
\cite{ShTs}.  Specifically, we will
make use of the following estimate, which as stated is from \cite{MS3}. 
\begin{lemma} \label{elliptic regularity lemma}
For $M,N = 0 , 1 , 2 , \ldots$ fixed and for $v \in C^\infty(\R \times
\ext)$ solving \eqref{perturbed wave} and vanishing for large $|x|$
for every $t$, it follows that
\begin{equation} \label{elliptic regularity ineq}
\begin{split}
	\sum_{|\mu| \leq N} \norm{L^M \partial^\mu v^\prime(t,\cdot)}{2} \lesssim 
	\ssum{m + j \leq N + M}{m \leq M} \bigl\| L^m \partial^j_t v^\prime(t,\cdot) \bigr\| _2 
	+ \ssum{|\mu| + m \leq N + M - 1}{m \leq M} \norm{L^m \partial^\mu \square v(t,\cdot)}{2}.
\end{split}
\end{equation}
\end{lemma}
To obtain useful estimates involving $L$, a modified version of this
operator, $\tilde{L}$, shall be introduced.  We set 
\begin{equation} \label{modified scaling vf}
\tilde{L} = t \partial_t + \chi(x) r \partial_r,
\end{equation}
 where $\chi(x) = 0$ for $x \in \K$ and $\chi(x) = 1$ for $|x| > 1$. It should be obvious that $\tilde{L}$ preserves the Dirichlet boundary conditions. However, due to the fact that $\tilde{L}$ fails to commute with $\square$, we must be able to control the commutator terms that arise in the next estimate. We fix the quantity
\[
X_{j,M}(t,x) = X_{j,M}(v)(t,x) = \int_{\ext} e_0(\tilde{L}^M \partial^j_t v)(t,x) \; dx.
\]
Using this modified energy quantity, we state the following lemma,
which is essentially from \cite{MS3}.  See, also, \cite{HM}.
\begin{lemma} \label{modified energy}
Let $v \in C^\infty (\R \times \ext)$ solve \eqref{perturbed wave}
where $\gamma^{\alpha \beta}$ are as above.  If $v(t,x)$ vanishes for large $|x|$ for each fixed $t$, then it follows that
\begin{multline} \label{modified energy ineq}
\partial_t X_{j,M}^{1/2} (t) \lesssim X_{j,M}^{1/2} (t,x) \norm{\gamma^\prime(t,\cdot)}{\infty} + \norm{\tilde{L}^M \partial^j_t \square_\gamma v (t,\cdot)}{2} + \norm{[\tilde{L}^M \partial^j_t  , \gamma^{\alpha \beta} \partial_\alpha \partial_\beta ] v (t,\cdot)}{2} \\
+ \sum_{m \leq M - 1} \norm{L^m \partial^j_t \square v (t,\cdot)}{2} + \ssum{m + |\mu| \leq M + j}{m \leq M-1} \norm{L^m \partial^\mu v^\prime(t,\cdot)}{L^2(|x| < 1)} .
\end{multline}
\end{lemma}

To shorten notation, we have written $L^2( \{ x \in \ext : |x| < 1 \} )$
as $L^2(|x| < 1)$.


\noindent

We now state our final energy estimate which involves the full
collection of admissible vector fields: scaling, rotations and
translations. This estimate follows from the same proof as Lemma
\ref{perturbed energy lemma} except that one applies the trace theorem
to the resulting boundary terms.  We will control the terms that arise from the boundary using \eqref{local energy} and localized energy estimates, which will appear in the next section of this paper.
%
\begin{prop} \label{perturbed energy full vector fields prop}
For fixed $N,M$, set
\[
Y_{N,M} (t) = \ssum{m + |\mu| \leq N + M}{m \leq M \\ |\nu| = 1} \int_{\ext} e_0(L^m Z^\mu \partial^\nu v)(t,x) \; dx . 
\]
Suppose that \eqref{perturbation small} 
holds for $\delta$ sufficiently small.  Also suppose that $v(t,x)$ vanishes for large $|x|$ for every $t$. Then it follows that
\begin{multline} \label{perturbed energy full vector fields}
\partial_t Y_{N,M} (t) \lesssim Y^{1/2}_{N,M} (t) \ssum{m + |\mu| \leq N + M}{m \leq M \\ |\nu|  = 1} \norm{\square_\gamma L^m Z^\mu \partial^\nu v(t,\cdot)}{2} \\
+ \norm{\gamma^\prime(t,\cdot)}{\infty} Y_{N,M} (t) + \ssum{m + |\mu| \leq N + M + 2}{m \leq M} \norm{L^m \partial^\mu v^\prime(t,\cd)}{L^2(|x| < 2)}^2 .
\end{multline}
\end{prop}



Just as in \cite{HM}, we note that this estimate slightly deviates 
from the versions appearing in earlier papers \cite{MS3,MS2}. The important difference is 
that in \cite{MS3,MS2}, one did not need to distinguish between $Z$ and a derivative 
$\partial$ in the definition of $Y_{N,M}$. Despite this subtle difference, 
the proof of the above proposition is identical to the proof presented in \cite{MS3}.

\subsubsection{Localized energy estimates and boundary term estimates on $\R\times\ext$}
This section concerns solutions to the Dirichlet-wave equation
\begin{equation} \label{boundary term wave}
	\left\{
	\begin{array}{l}
		\square v (t,x) = G (t,x) , \quad (t,x) \in \R \times \ext, \\
		v(t,x) = 0 , \quad x \in \partial \K, \\
		v(t,x) = 0 , \quad t \leq 0.
	\end{array}
	\right.
\end{equation}
Here we seek to provide estimates to handle the boundary terms that
arise in Lemma \ref{modified energy} and Proposition \ref{perturbed
  energy full vector
  fields prop}.

The first is a variant of the localized energy estimate
\eqref{energy} that holds in our exterior domains.
%
\begin{prop}
Suppose that $\K \subset \{ x \in \ext : |x| < 1 \}$ satisfies  \eqref{local energy} and suppose that $v \in C^\infty (\R \times \ext)$ solves \eqref{boundary term wave}. Then for any fixed $N$ and $0 \leq M \leq 1$, if $v(t,x)$ vanishes for large $|x|$ for every $t$, then
\begin{multline} \label{localized energy local control}
	\ssum{|\mu| + m \leq N + M}{m \leq M} \norm{L^m \partial^\mu v^\prime}{L^2_{t,x}([0,T] \times \{ x \in \ext : |x| < 5 \})} 
	\lesssim \int^T_0 \ssum{|\mu| + m \leq N + M + D}{m \leq M}\norm{\square L^m \partial^\mu v(s,\cdot)}{2} \; ds \\
	+ \ssum{|\mu| + m \leq N + M - 1}{m \leq M} \norm{\square L^m \partial^\mu v}{L^2_{t,x}([0,T] \times \ext)} .
\end{multline}
\end{prop}

\begin{proof}
Using cutoffs, we split our analysis into two cases: (1) $\square
v(t,x) = 0$ when $|x| < 10$, and (2) $ \square v (t,x) = 0$ when $|x|
> 6$. 

In the former case, we apply elliptic regularity \eqref{elliptic regularity ineq} and local energy decay \eqref{local energy} to see that
\begin{multline} \label{localized energy local control ineq1}
	\ssum{|\mu| + m \leq N + M}{m \leq M} \norm{L^m \partial^\mu v^\prime(t,\cdot)}{L^2( \{ x \in \ext : |x| < 5 \})}^2\\ 
	\lesssim \left( \int^t_0 \ssum{|\mu| + m \leq N + M + D}{m \leq M} \left< t - s \right>^{-2-\sigma + m} \norm{\square L^m \partial^\mu v(s,\cdot)}{L^2(\{ x \in \ext : |x| < 10 \})} \; ds \right)^2 \\
	+ \ssum{|\mu| + m \leq N + M - 1}{m \leq M} \norm{\square L^m \partial^\mu v(t,\cdot)}{L^2(\{ x \in \ext : |x| < 10 \})}^2 .
\end{multline}
See \cite[Lemma 2.8]{MS3} for more details.
Applying Young's inequality to the first term in the right hand side
of \eqref{localized energy local control ineq1} establishes
\eqref{localized energy local control} for case (1). 

To handle case (2), fix a cutoff $\rho \in C^\infty (\R^4)$ where $\rho(x) = 1$ when $|x| < 5$ and $\rho(x) = 0$ when $|x| > 6$. Let $w = \rho v_0 + v_r$ where $v_0$ solves the boundaryless wave equation $\square v_0  = \square v $ with vanishing Cauchy data.  We see that $w$ solves 
\[ \left\{ \begin{array}{l}
\square w (t,x) = -2 \nabla_x \rho(x) \cdot \nabla_x v_0 (t,x) -
(\Delta \rho(x)) v_0(t,x) , \quad (t,x) \in [0,T] \times \ext , \\ 
w(t,x)=0,\quad x\in \bdy,\\
w(0,x) = 0 , \quad t \leq 0.
\end{array}
\right.
\]
Note that $v(t,x) = w(t,x)$ for $|x| < 5$. 
Applying \eqref{localized energy local control ineq1} to $w$, integrating
both sides over $[0,T]$, and using Young's inequality, we see that

\begin{multline} \label{localized energy local control ineq 3}
\ssum{|\mu| + m \leq N + M}{m \leq M} \norm{L^m \partial^\mu w^\prime}{L^2( [0,T] \times \{ x \in \ext : |x| < 5 \})} \\
 \lesssim \ssum{|\mu| + m \leq N + M +D}{m \leq M} \norm{ L^m \partial^\mu u_0}{L^2([0,T ] \times \{ x \in \ext : |x| < 10 \})} \\
+ \ssum{|\mu| + m \leq N + M +D}{m \leq M} \norm{ L^m \partial^\mu u_0^\prime}{L^2([0,T ] \times \{ x \in \ext : |x| < 10 \})} .
\end{multline}
Applying Sobolev embedding and \eqref{main energy ineq} to right hand side of \eqref{localized energy local control ineq 3} finishes the proof.
\end{proof}

One of the major innovations of \cite{MS3} was an estimate that could
control the boundary term that results from the
commutator of $\Box$ with $\tilde{L}$.  Such was also used in the
three dimensional analog of the current study \cite{HM}.  This
original boundary term estimate was proved using Huygens' principle
for the free wave equation.  Here we instead use the analogous result
of \cite{MS2} for higher dimensions where the fundamental solution is
instead estimated.  We omit the proof as it is a straightforward
adaptation of \cite[Lemma 5.2]{MS2}, where the adaptation is in the
spirit of that from preceding proof.

\begin{prop} \label{boundary term prop}
Suppose that $v \in C^\infty (\R \times \ext )$ solves \eqref{boundary term wave} and vanishes for each $x \in \partial \K$. Also suppose that $\K \subset \{ |x| < 1 \}$ satisfies \eqref{local energy} and that $v(t,x)$ vanishes for large $|x|$ for every $t$. Then it follows that if $N \geq 0$ and $0 \leq M \leq 1$ are fixed, then we have the estimate
\begin{multline} \label{boundary term ineq}
\int^t_0 \ssum{|\mu| + m \leq N + M}{m \leq M} \norm{L^m \partial^\mu v^\prime(s,\cdot)}{L^2 (|x| < 2)} \; ds \lesssim \int^t_0 \ssum{|\mu| + m \leq N + M + D }{m \leq M} \norm{L^m \partial^\mu \square v (s,\cdot)}{2} \; ds \\
+ \int^t_0 \int_{\ext} \ssum{|\mu| + m \leq N + M + D+4}{m \leq M} \left| L^m Z^\mu \square v(s,y) \right| \; \dfrac{dy \; ds}{|y|^{3/2}} .
\end{multline} 
\end{prop}

\newsection{Pointwise Estimates}
We now state the main pointwise estimates that will be used to prove Theorem \ref{thm1}.
We first need a version of the Sobolev embedding theorem for annuli. See \cite{Klainerman2} for more details.
\begin{lemma} \label{sobolev on annuli}
Suppose that $h \in C^\infty(\R^n)$. Then it follows that for $R \geq 1$,
\begin{equation} \label{sobolev annuli 1}
\norm{h}{L^\infty(R/2 < |x| < R)} \lesssim R^{-(n-1)/2} \sum_{|\mu| \leq \frac{n}{2} + 1} \norm{Z^\mu h}{L^2(R/4 < |x| < 2R)} .
\end{equation}
\end{lemma}

After localizing to an annulus using a cutoff, these estimates follow from applying Sobolev embedding on $\R_+ \times \S^{n-1}$ and the fact that the volume element in $\R^n$ in polar coordinates is $r^{n-1} dr d\omega$.

We will also need estimates similar to those originally from \cite{KlainermanSideris}. Similar estimates have also appeared in \cite{Klainerman}, 
\cite{Sideris2}, \cite{SiderisThomases}, \cite{SiderisTu}, \cite{H3}
and \cite{HY}. Most versions of these estimates involve controlling
$u'$ whereas our current estimate uses Lemma \ref{DZ sobolev estimate}
to obtain a pointwise dispersive estimate for the solution $u$
itself. 

Our main pointwise estimate will be a combination of the estimates of
\cite{KlainermanSideris} and \cite{DMSZ}.  We begin with an estimate
that is essentially from \cite{H3} and is strongly rooted in the
preceding results of \cite{KlainermanSideris} and \cite{Sideris2}.
As opposed to what appeared in \cite{H3}, by merely having a spatial derivative on the right, we can eliminate
the $x$ dependence with the weight on the first term on the right.
\begin{prop} \label{klainerman sideris theorem}
Suppose $n \geq 3$. Let $v \in C^\infty (\R \times \R^n)$ such that for each fixed $t$, $v(t,x)$ vanishes for $|x|$ sufficiently large. Then it follows that
\begin{equation} \label{klainerman sideris}
\left< r \right>^{n/2 - 1} \left< t - r \right> \left| \nabla_x v(t,x) \right| \lesssim \left< t \right> \sum_{|\mu| \leq n/2 + 1} \norm{Z^\mu \square v (t,\cdot)}{2}  + \ssum{|\mu| + m \leq n/2 + 1}{m \leq 1} \norm{ L^m Z^\mu  v^\prime (t,\cdot)}{2} . 
\end{equation}
\end{prop}
In order to prove Proposition \ref{klainerman sideris theorem}, we state a couple of preliminary lemmas. 
\begin{lemma}[{\cite[Lemma 2.3]{KlainermanSideris}}] \label{Laplace lemma}
Let $v \in C^\infty(\R \times \R^n)$. Then it follows that
\begin{equation} \label{Laplacian estimate}
	\left< t - r \right> |\Delta v(t,x) | 
	\lesssim \sum_{|\alpha| + \mu \leq 1} \left| \partial L^\mu Z^\alpha v(t,x) \right| 
	+ t \left| \square v(t,x) \right|.
\end{equation}
\end{lemma}
The next estimate appeared in \cite{H3}. We state a variant of the original estimate in which only spatial derivatives are being applied to the solution in the left hand side.
\begin{lemma}[{\cite[Lemma 4.1]{H3}}] \label{hidano lemma}
Let $v \in C^\infty(\R \times \R^n)$. Then for $n \geq 3$, it follows that
\begin{equation} \label{hidano inequality}
\left< r \right>^{n/2 - 1} \left< t - r \right> | \nabla_x v (t,x)| \lesssim \ssum{|\mu| \leq n/2 + 1}{|\nu| = 2} \norm{\left< t - r \right>  \partial^\nu_x Z^\mu v (t,\cdot)}{2} + \sum_{|\mu| \leq n/2 + 1} \norm{Z^\mu \nabla_x v(t,\cdot)}{2} .
\end{equation}
\end{lemma}
\noindent
We are now ready to prove Proposition \ref{klainerman sideris theorem}.
\begin{proof}[Proof of Proposition \ref{klainerman sideris theorem}]
Applying Lemma \ref{hidano lemma}, one can see that it suffices to show that the first term in the right hand side of \eqref{hidano inequality} is controlled by the right hand side of \eqref{klainerman sideris}. Fixing $\mu$ and letting $w = Z^\mu v$, We note that integration by parts gives
\begin{equation} \label{klainerman sideris ineq 1}
\begin{split}
\sum_{i,j=1}^n \norm{\left< t - r \right> \partial_i \partial_j
  w(t,\cdot)}{2}^2 & = \sum_{i,j=1}^n \int \left< t -r
\right>^2 \partial_i \partial_j w \; \partial_i \partial_j w \; dx \\
&= \sum_{i,j=1}^n \int \left< t - r \right>^2 \partial_i^2 w \partial_j^2 w \; dx \\
& \quad + 2 \sum_{i,j=1}^n \int \left<t-r\right> \partial_i \left< t - r \right> \;  \partial_i w\partial_j^2 w \; dx \\
& \quad - 2 \sum_{i,j=1}^n \int \left< t - r \right>\partial_j  \left< t - r \right> \partial_i w \;  \partial_i \partial_j w \; dx .
\end{split}
\end{equation}
Applying Cauchy-Schwarz and the inequality $ab \leq (a^2+b^2)/2$, we see that the right hand side of \eqref{klainerman sideris ineq 1} is bounded by
\begin{align*}
C \norm{\left< t -r \right> \Delta w(t,\cdot)}{2}^2 + C \norm{\nabla_x w(t,\cdot)}{2}^2 + \dfrac{1}{4} \sum_{i,j=1}^n \norm{\left< t - r \right> \partial_i \partial_j w(t,\cdot)}{2}^2 ,
\end{align*}
where $C$ is sufficiently large. Applying Lemma \ref{Laplace lemma} to the first term in the above expression, we see that it is controlled by the right hand side of \eqref{klainerman sideris}. The second term is controlled by the second term in the right hand side of \eqref{klainerman sideris}. The last term can be bootstrapped back into the left hand side of \eqref{klainerman sideris ineq 1}. 
\end{proof}

Now we will combine Lemma \ref{DZ sobolev estimate} and Proposition \ref{klainerman sideris theorem} to prove our main pointwise estimate in $\R \times \R^n$.
\begin{prop}
Let $v \in C^\infty (\R \times \R^n)$ where $n \geq 3$ such that for each fixed $t$, $v(t,x)$ vanishes for $|x|$ sufficiently large. Then it follows that
\begin{multline} \label{klainerman sideris no gradient}
\left< r \right>^{n/2 - 1} \left< t - r \right> \left|  v(t,x) \right|  
\lesssim 
\ssum{|\mu| + m \leq n/2 + 1}{m \leq 1} \norm{L^m Z^\mu v(0,\cdot)}{2} + \ssum{|\mu| + m \leq n/2 + 1}{m \leq 1} \norm{L^m Z^\mu \partial_t v(0,\cdot)}{\dot{H}^{-1}}
\\ +\left< t \right> \sum_{|\mu| \leq n/2 + 1} \norm{ \partial^\mu \square v (t,\cdot)}{L^2 (|x| < 1)} 
 + \left< t \right> \sum_{|\mu| \leq n/2 + 1} \norm{|x|^{-(n-2)/2} Z^\mu \square v (t,\cdot)}{L^1_r L^2_\omega (|x| > 1)} \\ 
 + \int^t_0 \ssum{|\mu| + m \leq n/2 + 1}{m \leq 1} \norm{ L^m \partial^\mu  \square v(s,\cdot)}{L^2(|x| < 1)} \; ds \\
 + \int^t_0 \ssum{|\mu| + m \leq n/2 + 1}{m \leq 1} \norm{|x|^{-(n-2)/2} L^m Z^\mu \square v(s,\cdot)}{L^1_r L^2_\omega (|x| > 1)} \; ds. 
\end{multline}
\end{prop}
\begin{proof}
We proceed in a manner similar to the beginning of the proof of Theorem 2.3 in \cite{DMSZ}.  
We first note that for $1 \leq i,j,k \leq n$ and any $h \in C^\infty_0(\R \times \R^n)$ the following inequalities hold:
\begin{equation} \label{InverseLaplace}
\begin{split}
& \norm{\Delta^{-1}\partial_j h}{2} \lesssim \norm{h}{\dot{H}^{-1}(\R^n)}, \\
& \| [\Omega_{ij} , \Delta^{-1}\partial_k] h \|_2 \lesssim \sum^n_{\ell=1} \|\Delta^{-1}\partial_\ell h\|_2, \\
& \| [L , \Delta^{-1}\partial_j] h \|_2 \lesssim  \|\Delta^{-1}\partial_j h\|_2.
\end{split}
\end{equation}
We define
\[
v_j(t,x) = (2 \pi)^{-n/2} \int \dfrac{\xi_j}{|\xi|^2} e^{i x \cdot \xi} \; \widehat{v} (t,\xi) \; d \xi \]
where $\widehat{v}$ is the Fourier transform of $v$ in the $x$ variable. One can see that $i v = \sum^n_{j=1} \partial_j v_j .$ We then apply Proposition \ref{klainerman sideris theorem} and the energy inequality to see that
\begin{multline} \label{ks no grad ineq}
\left< r \right>^{n/2 - 1} \left< t - r \right> \left|  \partial_j v_j
  (t,x) \right| \lesssim   \ssum{|\mu| + m \leq n/2 + 1}{m \leq 1} \norm{ L^m Z^\mu v_j'(0,\cdot)}{2}
+\left< t \right> \sum_{|\mu| \leq n/2 + 1} \norm{Z^\mu \square v_j (t,\cdot)}{2} \\ 
 + \int^t_0 \ssum{|\mu| + m \leq n/2 + 1}{m \leq 1} \norm{ L^m Z^\mu  \square v_j (s,\cdot)}{2} \; ds.
\end{multline}
By our earlier observations \eqref{InverseLaplace}, we can apply Lemma \ref{DZ sobolev estimate} to to see that the right hand side of \eqref{ks no grad ineq} is controlled by the right hand side of \eqref{klainerman sideris no gradient}. This completes the proof.
\end{proof}
We will now extend Proposition \ref{klainerman sideris} to wave equations in exterior domains. Although this is stated for 4-dimensional domains, it is evident that similar arguments will yield estimates for higher dimensions.
\begin{prop} \label{klainerman sideris no gradient prop}
Let $v \in C^\infty(\R \times \ext)$ solve
\begin{equation} \label{KS wave}
	\left\{
	\begin{array}{l}
		\square v (t,x) = G (t,x) , \quad (t,x) \in \R \times \ext, \\
		v(t,x) = 0 , \quad x \in \partial \K, \\
		v(t,x) = 0 , \quad t \leq 0.
	\end{array}
	\right.
\end{equation}
Also suppose that for each fixed $t$, $v(t,x)$ vanishes for $|x|$ sufficiently large. 
Then, for $M$ and $|\mu| = N$ fixed, the following estimate holds
\begin{multline} \label{klainerman sideris no gradient ext domain}
\left< r \right> \left< t - r \right> \left| L^M Z^\mu  v(t,x) \right|
\lesssim 
\left< t \right> \ssum{|\nu| \leq  N + M + 3}{m \leq M} \norm{|x|^{-1} L^m Z^\nu \square v (t,\cdot)}{L^1_r L^2_\omega (|x| > 1)} \\ 
 + \int^t_0 \ssum{|\nu| + m \leq  N + M + 3}{m \leq M + 1} \norm{|x|^{-1} L^m Z^\nu \square v(s,\cdot)}{L^1_r L^2_\omega (|x| > 1)} \; ds \\
 + \int^t_0 \ssum{|\nu| + m \leq  N + M + 4}{m \leq M + 1} \norm{L^m \partial^\nu v^\prime(s,\cdot)}{L^2(|x| < 4)} \; ds .
\end{multline}
\end{prop}
\begin{proof}
When $|x| < 4$, one can apply Sobolev embedding to see that the left hand side of \eqref{klainerman sideris no gradient ext domain} is controlled by
\begin{equation} \label{klainerman sideris no gradient ext domain ineq 1}
(1+t) \ssum{|\nu| + m \leq N + M + 3}{m \leq M}  \norm{L^m \partial^\nu v(t,\cdot)}{L^2(|x| < 4)} .
\end{equation}
For $|x| > 4$, fix a cutoff $\eta \in C^\infty(\R^n)$ such that $\eta(x) = 1$ for $|x| > 4$ and vanishes for $|x| < 3$. If $v_0 = \eta v$, then it follows that $v_0$ solves the boundaryless wave equation 
\begin{equation} \label{KS boundaryless}
\left\{
\begin{array}{l}
\square v_0 = \eta G - 2 \nabla_x \eta \cdot \nabla_x v - (\Delta \eta) v,\\
v_0(t,x) = 0, \quad t \leq 0.
\end{array}
\right. 
\end{equation}
Applying the analogous Minkowski space estimate provided in 
\eqref{klainerman sideris no gradient}, we see that we obtain the inequality
\begin{multline} \label{klainerman sideris no gradient ext domain ineq 2}
\left< r \right> \left< t - r \right> \left| L^M Z^\mu  v_0(t,x)
\right|  \lesssim 
\left< t \right> \ssum{|\nu| \leq N + M + 3}{m \leq M} \norm{|x|^{-1} L^m Z^\nu \square v (t,\cdot)}{L^1_r L^2_\omega (|x| > 1)} \\ 
 + \int^t_0 \ssum{|\nu| + m \leq  N + M + 3}{m \leq M + 1} \norm{|x|^{-1} L^m Z^\nu \square v(s,\cdot)}{L^1_r L^2_\omega (|x| > 1)} \; ds \\
 + \la t\ra \ssum{|\nu| + m \leq  N + M + 4}{m \leq M} \norm{L^m \partial^\nu v (t,\cdot)}{L^2(|x| < 4)} \\
 + \int^t_0 \ssum{|\nu| + m \leq  N + M + 4}{m \leq M + 1} \norm{L^m \partial^\nu v(s,\cdot)}{L^2(|x| < 4)} \; ds. 
\end{multline}
Applying the Fundamental Theorem of Calculus and \eqref{local control}, we see that \eqref{klainerman sideris no gradient ext domain ineq 1} and the last two terms in \eqref{klainerman sideris no gradient ext domain ineq 2} are controlled by
\[
\int^t_0 \ssum{|\nu| + m \leq  N + M + 4}{m \leq M + 1} \norm{L^m \partial^\nu v^\prime(s,\cdot)}{L^2(|x| < 4)} \; ds.
\]
This completes the proof.
\end{proof}
\newsection{Proof of Theorem \ref{thm1}}
We will now prove Theorem \ref{thm1} using an iteration argument. Using standard local existence theory (see \cite[Theorems 9.4, 9.5]{KSS2} for more details), we first note that we have a local solution on a fixed timestrip.
\begin{theorem} \label{local existence thm}
	If $\epsilon$ in \eqref{lifespan} is sufficiently small and $f,g$ are as in Theorem \ref{thm1} with $N \geq 9$, then there is a local in time solution $u \in C^\infty ([0,2] \times \ext )$ to \eqref{main} that satisfies
	\begin{equation} \label{local existence}
		\sup_{0 \leq t \leq 2 } \sum_{|\mu| \leq N} \norm{\partial^\mu u(t,\cdot)}{2} \leq C \epsilon.
	\end{equation}
\end{theorem}
Using the fact that we have a local solution on $[0,2] \times \ext$, we will now reduce to the case where we have vanishing initial data at the expense of an additional forcing term in our nonlinear equation. This will enable us to avoid dealing with the compatibility conditions on the initial data $(f,g)$ that were mentioned earlier. To do this, we first fix a cutoff in time $\eta \in C^\infty(\R)$ such that $\eta(t) = 1$ for $t < 1/2$ and $\eta(t) = 0$ for $t > 1$. If we set $u_0 = \eta u$, then $u_0$ solves
\[
	\square u_0 = \eta Q (u,u^\prime,u^{\prime \prime}) + [ \square, \eta] u.
\] 
Letting $w = u - u_0$, it is clear that solving our original equation \eqref{main} is equivalent to showing that $w$ solves
\begin{equation} \label{zero data} \left\{
	\begin{array}{l}
		\square w = ( 1 - \eta ) Q (u_0 + w,(u_0 + w)^\prime,(u_0 + w)^{\prime \prime}) - [ \square, \eta] u, \\
		w(t,x) = 0 , \quad x \in \partial \K, \\
		w(0,x) = \partial_t w(0,x) = 0.
	\end{array} \right.
\end{equation}
We solve our new equation \eqref{zero data} using an iteration argument. We set the initial term $w_{0} \equiv 0$. We then recursively define $w_k$ to solve
\begin{equation*} \label{iteration} \left\{
	\begin{array}{l}
		\square w_k = ( 1 - \eta ) Q (u_0 + w_{k-1},(u_0 + w_{k-1})^\prime,(u_0 + w_k)^{\prime \prime}) - [ \square, \eta] u, \\
		w_k(t,x) = 0 , \quad x \in \partial \K, \\
		w_k(0,x) = \partial_t w_k(0,x) = 0.
	\end{array} \right.
\end{equation*}
We first show that our solution is bounded in an appropriate norm. To construct this norm, we fix an integer $N_0$ with the property that
\[
N_0 \geq \dfrac{N_0 + 6D + 61}{2} + 10,
\]
where $D$ is the integer appearing in \eqref{local energy}. This inequality will be used implicitly throughout the iteration argument to control the lower order terms that result from applying the product rule.
For each $k$, we set
\begin{multline} \label{solution norm}
	M_k(T) = \sup_{0 \leq t \leq T} \sum_{|\mu| \leq N_0 + 6D + 60} \norm{\partial^\mu w_k^\prime(t,\cdot)}{2} + \sum_{|\mu| \leq N_0 + 5D + 50} \norm{\left< x \right>^{-3/4} \partial^\mu w_k^\prime}{L^2_{t,x}(S_T)} \\
	+ \sup_{0 \leq t \leq T} \ssum{|\mu| \leq N_0 + 4D + 40}{|\nu| \leq 2} \norm{Z^\mu \partial^\nu w_k(t,\cdot)}{2} + \log(2 + T)^{-1/2} \ssum{|\mu| \leq N_0 + 4D + 40}{|\nu| \leq 1} \norm{\left< x \right>^{-1/2} Z^\mu \partial^\nu w_k}{L^2_{t,x}(S_T)} \\
	+ \sup_{0 \leq t \leq T} \sum_{|\mu|\leq N_0 + 3D + 30} \norm{L \partial^\mu w_k^\prime(t,\cdot)}{2} + \sum_{|\mu| \leq N_0 + 2D + 20} \norm{\left< x \right>^{-3/4} L \partial^\mu w_k^\prime}{L^2_{t,x}(S_T)} \\
	+ \sup_{0 \leq t \leq T} \ssum{|\mu| \leq N_0 + D + 10}{|\nu| \leq 2} \norm{L Z^\mu \partial^\nu w_k(t,\cdot)}{2} + \log(2 + T)^{-1/2} \ssum{|\mu| \leq N_0 + D + 10}{|\nu| \leq 1} \norm{\left< x \right>^{-1/2} L Z^\mu \partial^\nu w_k}{L^2_{t,x}(S_T)} \\
	+ \sup_{0 \leq t \leq T}  \sum_{|\mu| \leq N_0} \norm{ \left< r \right> \left< t - r \right> Z^\mu w_k (t,\cdot)}{\infty} .
\end{multline}
We label the terms in $M_k(T)$ by $ I , II , \ldots , IX.$ Letting $M_0(T)$ equal the above quantity with $w_k$ replaced by $u_0$, if we apply \eqref{local existence} with $N = N_0 + 6D + 62$, then it follows that there is a uniform constant $C_0$ such that
\begin{equation} \label{base case}
	M_0(T) \leq C_0 \epsilon .
\end{equation} 
Moreover, it also follows that
\begin{equation} \label{local solution bound}
\sup_{0 \leq t \leq T} \sum_{|\mu| \leq N_0 + 6D + 62} \norm{\partial^\mu u_0 (t,\cdot)}{2} \leq C_0 \epsilon.
\end{equation}
We wish to show via induction that there is a uniform constant $C_1$ such that
\begin{equation} \label{induction step}
	M_k(T) \leq C_1 \epsilon .
\end{equation}
Letting
\begin{equation} \label{induction hypothesis}
M_{k-1}(T) \leq C_1 \epsilon
\end{equation}
be our induction hypothesis, we will show that this implies \eqref{induction step} for $\epsilon$ taken to be sufficiently small. By induction, this will prove our desired uniform bound.
\par
{\em{\bf{Bound for $I$: }}}
We apply \eqref{perturbed energy ineq} and \eqref{elliptic regularity ineq} with $M = 0$ where the perturbation terms are set as
\begin{equation} \label{term1 perturbation terms}
	\gamma^{\alpha \beta} = - ( 1-  \eta) \left[ b^{\alpha \beta} (u_0 + w_{k-1}) + b^{\alpha \beta \gamma} \partial_\gamma (u_0 + w_{k-1}) \right] .
\end{equation}
By the induction hypothesis \eqref{induction hypothesis} and the bound given to us by the local existence theorem \eqref{local existence}, we know that \eqref{perturbation small} and \eqref{gamma prime small} hold with $\delta = C_1 \epsilon$. Applying Gronwall's inequality to \eqref{perturbed energy ineq} and applying \eqref{lifespan}, we see that an application of \eqref{elliptic regularity ineq} shows that controlling $I$ reduces to proving bounds for
\begin{equation} \label{term1 reduction}
	\int^T_0 \sum_{j \leq N_0 + 6D + 60} \norm{\square_\gamma \partial^j_t w_k(t,\cdot)}{2} \; dt 
	+ \sup_{0 \leq t \leq T} \sum_{|\mu| \leq N_0 + 6D + 59} \norm{\partial^\mu \square w_k(t,\cdot)}{2} .
\end{equation}
It is clear that
\begin{multline} \label{term1 perturbation ineq}
	\sum_{j \leq N_0 + 6D + 60} | \square_\gamma \partial^j_t w_k | 
	\lesssim \sum_{|\mu| \leq N_0} | \partial^\mu ( u_0 + w_{k-1} ) |  \\
	\times \Bigl[ \sum_{|\mu| \leq N_0 + 6D + 60} | \partial^\mu (u_0 + w_k)^\prime | + \sum_{|\mu| \leq N_0 + 6D + 62} | \partial^\mu u_0  | \Bigr] \\
	+ \sum_{|\mu| \leq N_0 - 1} | \partial^\mu ( u_0 + w_{k} )^\prime | \sum_{|\mu| \leq N_0 + 6D + 60} | \partial^\mu ( u_0 + w_{k-1} )^\prime | \\
	+ \sum_{|\mu| \leq N_0 - 1} | \partial^\mu ( u_0 + w_{k-1} )^\prime | \sum_{|\mu| \leq N_0 + 6D + 60} | \partial^\mu ( u_0 + w_{k-1} )^\prime | \\
	+ | u_0 + w_{k-1} |^2 + \sum_{|\mu| \leq N_0 + 6D + 60} | \partial^\mu [ \square , \eta ] u | .
\end{multline}
By \eqref{term1 perturbation ineq} and \eqref{local solution bound}, it follows that the first term in \eqref{term1 reduction} can be controlled using terms $I$, $III$ and $IX$ in \eqref{solution norm}. It follows from \eqref{local existence} that
\begin{multline} \label{term1 controlling gamma term}
	\int^T_0 \sum_{j \leq N_0 + 60 + 6M} \norm{\square_\gamma \partial^j_t w_k(t,\cdot)}{2} \; dt \lesssim ( M_0 (T) + M_k(T)) (M_0(T) + M_{k-1}(T)) \int^T_0 (1 + t)^{-1} \; dt \\
	+ (M_0(T) + M_{k-1}(T))^2 \int^T_0 (1 + t)^{-1} \; dt + \epsilon.
\end{multline}
A similar argument shows that the second term in \eqref{term1 reduction} is also controlled by the right hand side of \eqref{term1 controlling gamma term}. It follows that
\begin{multline}
\label{term1 bound}
	I \leq C ( M_0 (T) + M_{k-1}(T) + M_k(T)) (M_0(T) 
	+ M_{k-1}(T)) \log( 2 + T)  + C_2 \epsilon,
\end{multline}
where $C_2$ is a constant that is chosen to be sufficiently large relative to the size of the constant $C$ in \eqref{local existence}. For the remainder of this proof, we will allow $C_2$ to vary from line to line, but it will be clear from the proof that this constant is independent of important parameters such as $k, \epsilon, T , C_1$.
\par
{\em{\bf{Bound for $II$: }}}
Fix a cutoff $\chi \in C^\infty_c(\R^4)$ such that $\chi(x) = 1 $ for $|x| < 3$ and zero for $|x| > 4$. Fixing the multi-index $\mu$, we will first consider $( 1- \chi) \partial^\mu w_k$. We see that this solves the boundaryless wave equation
\begin{equation} \label{term 2 cutoff equation}
\square (1 - \chi) \partial^\mu w_k = (1 - \chi) \partial^\mu \square w_k - [ \square , \chi ] \partial^\mu w_k  
\end{equation}
with vanishing Cauchy data.
We apply \eqref{main energy ineq} to $(1 - \chi) \partial^\mu w_k$ to see that
\begin{equation} \label{term2 reduction1}
II \lesssim \int^T_0 \sum_{|\mu| \leq N_0 + 5D + 50} \norm{\partial^\mu \square w_k (s,\cdot)}{2} \; ds + \sum_{|\mu| \leq N_0 + 5D + 50} \norm{\partial^\mu w_k^\prime}{L^2_{t,x} ( S_T \cap \{ |x| < 4 \} ) } ,
\end{equation}
where we have also applied \eqref{local control} to the commutator term in \eqref{term 2 cutoff equation}. If we apply \eqref{localized energy local control} to the second term on the right hand side of \eqref{term2 reduction1}, we see that controlling $II$ reduces to bounding
\begin{equation} \label{term2 reduction2}
\int^T_0 \sum_{|\mu| \leq N_0 + 6D + 50} \norm{\partial^\mu \square w_k (s,\cdot)}{2} \; ds + \sum_{|\mu| \leq N_0 + 5D + 49} \norm{\partial^\mu \square w_k}{L^2_{t,x} (S_T ) } .
\end{equation}
Using the fact that
\begin{multline} \label{term2 ineq}
	\sum_{|\mu| \leq N_0 + 6D + 50} | \square \partial^\mu w_k | 
	\lesssim \sum_{|\mu| \leq N_0} | \partial^\mu ( u_0 + w_{k-1} ) |  
\sum_{|\mu| \leq N_0 + 6D + 50} | \partial^\mu (u_0 + w_k)^{\prime \prime} |  \\
	+ \sum_{|\mu| \leq N_0 - 1} | \partial^\mu ( u_0 + w_{k} )^\prime | \sum_{|\mu| \leq N_0 + 6D + 50} | \partial^\mu ( u_0 + w_{k-1} )^\prime | \\
	+ \sum_{|\mu| \leq N_0 - 1} | \partial^\mu ( u_0 + w_{k-1} )^\prime | \sum_{|\mu| \leq N_0 + 6D + 50} | \partial^\mu ( u_0 + w_{k-1} )^\prime | \\
	+ | u_0 + w_{k-1} |^2 + \sum_{|\mu| \leq N_0 + 6D + 50} | \partial^\mu [ \square , \eta ] u | ,
\end{multline}
it follows that the first term in \eqref{term2 reduction2} can be
controlled using \eqref{local existence} and $I$, $III$ and $IX$ in \eqref{solution norm}:
\begin{multline}
\int^T_0 \sum_{|\mu| \leq N_0 + 5D + 51} \norm{\partial^\mu \square w_k(s,\cdot)}{2}\; ds \lesssim ( M_0 (T) + M_k(T)) (M_0(T) + M_{k-1}(T)) \int^T_0 (1 + t)^{-1} \; dt \\
	+ (M_0(T) + M_{k-1}(T))^2 \int^T_0 (1 + t)^{-1} \; dt + \epsilon.
\end{multline}
Since the second term in \eqref{term2 reduction2} can be controlled using a simpler argument, it follows that
\begin{multline} \label{term2 bound}
II \leq C ( M_0 (T) + M_{k-1}(T) + M_k(T)) (M_0(T) + M_{k-1}(T)) \log( 2 + T)  + C_2 \epsilon .
\end{multline}
\par
{\em{\bf{Bound for $III$ and $IV$, $|\nu| = 0$: }}}
To control most of these terms, we will be applying Proposition \ref{u bound prop}. The only case in which this method will not work is the instance where the full number of vector fields  is applied to the terms with second-order derivatives, which does not fit nicely with the weighted $L^2_{t,x}$ spaces that are specified in \eqref{solution norm}. We overcome this difficulty by expressing these terms in divergence form and applying Proposition \ref{divergence form prop}. 
\par
Due to \eqref{local control} 
and the fact that
\[
\norm{Z w_k(t,\cdot)}{L^2_x(\{ x \in \ext : |x| < 4 \} )} \lesssim \norm{w_k^\prime (t,\cdot)}{2}, 
\]
on the set $\{ |x| < 4 \}$, terms $III$ and $IV$ are controlled by $I$ and $II$. Thus, to complete the bound for terms $III$ and $IV$, it suffices to control $( 1 - \chi ) Z^\mu w_k $, where $|\mu| \leq N_0 + 4D + 40$ is fixed and $\chi$ is the same cutoff function that was defined earlier. We note that $( 1 - \chi ) Z^\mu w_k $ solves the boundaryless wave equation,
\begin{equation} \label{terms 3 and 4}
\square ( 1 - \chi) Z^\mu w_k =  (1 - \chi) Z^\mu \square w_k - [\square , \chi] Z^\mu w_k,
\end{equation}
where we note that the second term in the right hand side is supported on the set $\{ |x| < 4 \}$. We further rewrite the first term as
\begin{multline} \label{terms3and4decomposition}
(1 - \chi) Z^\mu \square w_k = (1 - \eta)(1 - \chi) \partial_\alpha \bigl[ (1 - \eta ) (1 - \chi) \bigl( b^{\alpha \beta} (u_0 + w_{k-1}) Z^\mu \partial_\beta (u_0 + w_{k})  \\
 + b^{\alpha \beta \gamma} \partial_\gamma (u_0 + w_{k-1}) Z^\mu \partial_\beta (u_0 + w_{k}) \bigr) \bigr] + G_{k,\mu}(t,x),
\end{multline}
where $ G_{k,\mu}(t,x)$ contains no terms that have more
than a total of $|\mu|+1$ derivatives and vector fields applied to them. Applying \eqref{divergence form} and \eqref{u bound} to the first and second terms in \eqref{terms3and4decomposition}, respectively, and \eqref{cutoff estimate} to the commutator term in the right hand side of \eqref{terms 3 and 4} we see that
\begin{multline} \label{terms3and4 ineq 1}
\sup_{t \in [0,T]} \sum_{|\mu| \leq N_0 + 4D + 40} \norm{(1 - \chi)Z^\mu w_k(t,\cdot)}{2} \\
+ \sum_{|\mu| \leq N_0 + 4D + 40} \log( 2 + T)^{-1/2} \norm{\left< x \right>^{-1/2} (1 - \chi)Z^\mu w_k(t,\cdot)}{L^2_{t,x}(S_T)} \\
 \lesssim 
\int^T_0  \sum_{|\mu|\le N_0 + 4D+40} \norm{ \left< x \right>^{-1} G_{k,\mu}(s,\cdot)}{L^1_r L^2_\omega (|x| > 3)}\; ds \\
 + \int^T_0 \sum_{|\theta| \leq 1} \sum_{|\mu| \leq N_0 + 4D + 40} \norm{\left| \partial^\theta (u_0 + w_{k-1}) \right| \left|Z^\mu (u_0 + w_k)^\prime \right|}{2} \; ds \\
+ \sum_{|\mu| \leq N_0 + 4D + 41} \norm{\partial^\mu w_k^\prime}{L^2_{t,x}(S_T \cap \{ |x| < 4 \})}.
\end{multline}
The last term in the right hand side of \eqref{terms3and4 ineq 1} is controlled by term $II$, whose bounds we established earlier. To control the first term in the right hand side of \eqref{terms3and4 ineq 1}, we see that
\begin{multline}
\sum_{|\mu|\le N_0+4D+40} |G_{k,\mu} | \lesssim \sum_{|\mu| \leq N_0} | Z^\mu ( u_0 + w_{k-1} ) | \sum_{|\mu| \leq N_0 + 4D + 40 } | Z^\mu (u_0 + w_k)^\prime |  \\
	+ \sum_{|\mu| \leq N_0} | Z^\mu ( u_0 + w_{k} )^\prime | \ssum{|\mu| \leq N_0 + 4D + 40}{|\lambda| \leq 1} | Z^\mu \partial^\lambda ( u_0 + w_{k-1} ) | \\
	+ \sum_{|\mu| \leq N_0} | Z^\mu ( u_0 + w_{k-1} ) | \ssum{|\mu| \leq N_0 + 4D + 40}{|\lambda| \leq 1} | Z^\mu \partial^\lambda ( u_0 + w_{k-1} ) | \\
+ \sum_{|\mu| \leq N_0 + 4D + 40} | \partial^\mu [ \square , \eta ] u | .
\end{multline} 
Applying \eqref{local existence}, Cauchy-Schwarz, and Sobolev embedding on $S^3$, we see that
\begin{multline}\label{423}
\int^T_0  \sum_{|\mu|\le N_0 + 4D + 40} \norm{\left< x \right>^{-1} G_{k,\mu} (s,\cdot)}{L^1_r L^2_\omega(|x| > 2)} \; ds \lesssim \epsilon \\
+ \sum_{|\theta| \leq N_0 + 2} \norm{ \left< x \right>^{-1/2} Z^\theta ( u_0 + w_{k-1} ) }{L^2_{t,x}(S_T)}  \sum_{|\theta| \leq N_0 + 4D + 40 } \norm{ \left< x \right>^{-1/2} Z^\theta  (u_0 + w_k)^\prime }{L^2_{t,x}(S_T)}  \\
	+ \sum_{|\theta| \leq N_0 + 2} \norm{ \left< x \right>^{-1/2} Z^\theta ( u_0 + w_{k} )^\prime }{L^2_{t,x}(S_T)}  \ssum{|\theta| \leq N_0 + 4D + 40}{|\lambda| \leq 1} \norm{ \left< x \right>^{-1/2} Z^\theta \partial^\lambda ( u_0 + w_{k-1} ) }{L^2_{t,x}(S_T)}  \\
	+ \sum_{|\theta| \leq N_0 + 2} \norm{ \left< x \right>^{-1/2} Z^\theta  ( u_0 + w_{k-1} )}{L^2_{t,x}(S_T)}  \ssum{|\theta| \leq N_0 + 4D + 40}{|\lambda| \leq 1} \norm{ \left< x \right>^{-1/2} Z^\theta \partial^\lambda ( u_0 + w_{k-1} ) }{L^2_{t,x}(S_T)} .
\end{multline}
Each factor in the right hand side can be bounded in terms of  $\log ( 2 + T )^{1/2} IV$. This implies that the right hand side of \eqref{423} is controlled by
\[
C (M_0(T) + M_{k-1}(T) + M_k(T)) (M_0(T) + M_{k-1}(T)) \log(2 +T ) + C_2 \epsilon.
\] 
The second term in the right hand side of \eqref{terms3and4 ineq 1} can be bounded using simpler arguments similar to those used to bound term $I$. Hence it follows that
\[
III \bigr|_{|\nu| = 0} + IV \bigr|_{|\nu| = 0} \leq C (M_0(T) + M_{k-1}(T) + M_k(T)) (M_0(T) + M_{k-1}(T)) \log(2 +T ) + C_2 \epsilon.
\]
\par
{\em{\bf{Bound for $III$ and $IV$, $|\nu| = 1$: }}}
As with the previous section, we will only concern ourselves with
controlling these terms away from the boundary. We fix the same cutoff
$\chi$ and observe that for $\{ |x| < 4 \}$, the coefficients of $Z$
are bounded. Thus, near the obstacle, we see that these terms can be
controlled by terms $I$ and $II$, and the previously established
bounds for these pieces can be applied.  It then suffices to consider
$(1 - \chi) Z^\mu \partial w_k$, where the multi-index $\mu$ is
fixed. Applying \eqref{energy} and \eqref{main energy ineq}, we see that
\begin{multline} \label{terms3and4 ineq 2} 
	\sup_{t \in [0,T]} \norm{Z^\mu \partial w_k (t,\cdot)}{L^2(|x| > 3)} 
	+ \log (2 + T)^{-1/2} \norm{\left< x \right>^{-1/2} Z^\mu \partial w_k}{L^2_{t,x} (S_T \cap \{ |x| > 3 \})} \\
 	\lesssim \sum_{|\theta| \leq N_0 + 4D + 40} \int^T_0 \norm{Z^\theta \square w_k (s,\cdot)}{2} \; ds 
 	+ \sum_{|\theta| \leq N_0 + 4D + 40} \norm{\partial^\theta w_k^\prime}{L^2_{t,x}(S_T \cap \{ |x| < 4 \} )} .
\end{multline}
As stated earlier, the last term in \eqref{terms3and4 ineq 2} is controlled by term $II$. To control the first term, we apply an argument similar to the one used to bound term $I$ to see that
\begin{multline}
\int^T_0 \sum_{|\theta| \leq N_0 + 4D + 40} \norm{Z^\theta \square w_k(s,\cdot)}{2} \; ds \lesssim\\ \int^T_0 \sum_{|\theta| \leq N_0} \norm{ Z^\theta ( u_0 + w_{k-1} )(s,\cdot) }{\infty}  
	\ssum{|\theta| \leq N_0 + 4D + 40}{|\lambda| = 2}
        \norm{ Z^\theta \partial^\lambda (u_0 + w_k)(s,\cdot) }{2} 
\; ds \\
	+ \int^T_0 \sum_{|\theta| \leq N_0 - 1} \norm{ Z^\theta ( u_0 + w_{k} )^\prime (s,\cdot) }{\infty} \sum_{|\theta| \leq N_0 + 4D + 40} \norm{ Z^\theta ( u_0 + w_{k-1} )^\prime (s,\cdot) }{2} \; ds \\
	+ \int^T_0 \sum_{|\theta| \leq N_0} \norm{ Z^\theta ( u_0 +
          w_{k-1} ) (s,\cdot)}{\infty} \sum_{\substack{|\theta| \leq
            N_0 + 4D + 40\\|\nu|\le 1}} \norm{ Z^\theta \partial^\nu( u_0 + w_{k-1} ) (s,\cdot) }{2} \; ds \\
+ \sup_{t \in [0,2]} \sum_{|\theta| \leq N_0 + 4D + 43} \norm{ \partial^\theta [ \square , \eta ] u (t,\cdot)}{2} .
\end{multline}
Using \eqref{local existence} and terms $III$ and $IX$ in \eqref{solution norm}, we see that the right hand side is controlled by
\[
C (M_0(T) + M_{k-1}(T) + M_k(T)) (M_0(T) + M_{k-1}(T)) \log(2 +T ) + C_2 \epsilon.
\]
\par
{\em{\bf{Bound for $III$, $|\nu| = 2$: }}} 
Just as in controlling term $I$, we choose $\gamma$ as before so that \eqref{perturbation small} and \eqref{gamma prime small} are satisfied with $\delta = C_1 \epsilon$. We then apply \eqref{perturbed energy full vector fields} and integrate in the $t$ variable. After applying \eqref{gamma prime small}, \eqref{lifespan} and Gronwall's inequality, we see that it suffices to control
\begin{equation} \label{terms3and4 ineq 3}
\int^T_0 \ssum{|\mu| \leq N_0 + 4D + 40 }{ |\lambda| = 1} \norm{\square_\gamma Z^\mu \partial^\lambda w_k(s,\cdot)}{2} \; ds + \sum_{|\mu| \leq N_0 + 4D + 42} \norm{\partial^\mu \square w_k}{L^2_{t,x} (S_T \cap \{ |x| < 2 \})} .
\end{equation}
The second term in \eqref{terms3and4 ineq 3} is controlled using the same argument used to bound term $II$. Thus, it remains to control the first term. We see that
\begin{multline} \label{terms3and4 ineq 4}
\ssum{|\mu| \leq N_0 + 4D + 40 }{ |\lambda| = 1} \left| \square_\gamma Z^\mu \partial^\lambda w_k \right| \lesssim \sum_{|\mu| \leq N_0} | Z^\mu ( u_0 + w_{k-1} ) |  \\
	\times \left( \ssum{|\mu| \leq N_0 + 4D + 40 }{|\lambda| \leq 1} | Z^\mu \partial^\lambda (u_0 + w_k)^\prime | + \sum_{|\mu| \leq N_0 + 4D + 43} |Z^\mu u_0| \right) \\
	+ \sum_{|\mu| \leq N_0 - 1} | Z^\mu ( u_0 + w_{k} )^\prime | \ssum{|\mu| \leq N_0 + 4D + 40}{|\lambda| \leq 2} | Z^\mu \partial^\lambda ( u_0 + w_{k-1} ) | \\
	+ \sum_{|\mu| \leq N_0} | Z^\mu ( u_0 + w_{k-1} ) | \ssum{|\mu| \leq N_0 + 4D + 40}{|\lambda| \leq 2} | Z^\mu \partial^\lambda ( u_0 + w_{k-1} ) | \\
+ \sum_{|\mu| \leq N_0 + 4D + 41} | \partial^\mu [ \square , \eta ] u | . 
\end{multline}
By arguments similar to those that were used to obtain \eqref{term1 controlling gamma term}, we can control the first term in \eqref{terms3and4 ineq 3} with \eqref{terms3and4 ineq 4}, \eqref{local existence} and terms $III$ and $IX$ from \eqref{solution norm}. This gives us the bound
\begin{multline} \label{terms3and4 ineq 5}
	\int^T_0 \ssum{|\mu| \leq N_0 + 4D + 40 }{ |\lambda| = 1} \norm{\square_\gamma Z^\mu \partial^\lambda w_k(s,\cdot)}{2} \; ds \\
	\lesssim ( M_0 (T) + M_k(T) + M_{k-1}(T)) (M_0(T) + M_{k-1}(T)) \int^T_0 (1 + t)^{-1} \; dt  + \epsilon.
\end{multline}
From this, it follows that
\begin{multline}
\label{term3 2 derivatives bound}
	III \bigr|_{|\nu| = 2} \leq C ( M_0 (T) + M_{k-1}(T) + M_k(T)) (M_0(T) 
	+ M_{k-1}(T)) \log( 2 + T)  + C_2 \epsilon.
\end{multline}
\par
{\em{\bf{Bound for $V$: }}}
In previous papers, such as \cite{DMSZ} and \cite{MS5}, the authors were able to avoid using $L$ in their iteration arguments. By assuming $\K$ is star-shaped, the authors of \cite{DMSZ} and \cite{MS5} were able to employ localized energy estimates for variable coefficient wave equations enabled them to avoid using $L$ in their existence arguments. However, the proof of these localized estimates relies explicitly on $\K$ being star-shaped to control the boundary terms that arise from integrating by parts. At this time, the authors are not aware of a proof of these estimates that relies solely on weaker geometric assumptions, such as \eqref{local energy}. Instead we follow an approach similar to \cite{HM} and \cite{MS3,MS2}, which relies on \eqref{boundary term ineq}.
\par
We use our elliptic regularity estimate \eqref{elliptic regularity ineq} to see that controlling $V$ reduces to proving bounds for
\[
 \ssum{j + m \leq N_0 + 3D + 31}{m \leq 1} \norm{L^m \partial_t^j w_k^\prime(t,\cdot)}{2} +  \ssum{|\mu| + m \leq N_0 + 3D + 30}{m \leq 1} \norm{L^m \partial^\mu \square w_k(t,\cdot)}{2} ,
\]
for $ 0 \leq t \leq T $. 
Since the energy estimates we have stated in this paper involve 
the modified scaling vector field defined in \eqref{modified scaling vf}, 
we note that the first term is controlled by
\begin{equation} \label{term5 reduction 1}
\sum_{j \leq N_0 + 3D + 30} \norm{\tilde{L} \partial_t^j w_k^\prime(t,\cdot)}{2} + \sum_{|\mu| \leq N_0 + 3D + 31} \norm{\partial^\mu w_k^\prime(t,\cdot)}{2} .
\end{equation}
The bounds for $I$ can be cited to handle the second term in
\eqref{term5 reduction 1}.
It follows that we can
further reduce
 controlling term $V$ to estimating
\begin{equation} \label{term5 reduction 2}
\sum_{j \leq N_0 + 3D + 30} \norm{\tilde{L} \partial_t^j w_k^\prime(t,\cdot)}{2} +  \ssum{|\mu| + m \leq N_0 + 3D + 30}{m \leq 1} \norm{L^m \partial^\mu \square w_k(t,\cdot)}{2}.
\end{equation}
We then let $\gamma$ be as in \eqref{term1 perturbation terms} and
apply \eqref{modified energy ineq} to the first term. After
integrating over the timestrip $[0,T]$ and applying \eqref{gamma prime small} and \eqref{lifespan}, we have \eqref{term5 reduction 2} bounded by
\begin{multline} \label{term5 reduction 3}
\int^T_0 \ssum{j + m \leq N_0 + 3D + 31}{m \leq 1} \Bigl[ \norm{ \tilde{L}^m \partial^j_s \square_\gamma w_k(s,\cdot)}{2}
+ \norm{ \left[ \tilde{L}^m \partial^j_s , \gamma^{\alpha \beta} \partial_\alpha \partial_\beta \right] w_k(s,\cdot)}{2}\Bigr] \; ds \\
+ \int^T_0 \Bigl[ \sum_{j \leq N_0 + 3D + 31} \norm{\partial^j_s \square w_k(s,\cdot)}{2} + \sum_{|\mu| \leq N_0 + 3D + 31} \norm{\partial^\mu w_k^\prime(s,\cdot)}{L^2(|x| < 1)} \Bigr] \; ds \\
+ \sup_{t \in [0,T]} \ssum{|\mu| + m \leq N_0 + 3D + 30}{m \leq 1} \norm{L^m \partial^\mu \square w_k (t,\cdot)}{2}.
\end{multline}
The third term in \eqref{term5 reduction 3} is controlled using the same arguments used to control term $I$. Using the product rule, we see that
\begin{multline} \label{term5 ineq 1}
\ssum{j + m \leq N_0 + 3D + 31}{m \leq 1} \bigl[ \bigl| \tilde{L}^m \partial^j_s \square_\gamma w_k \bigr| 
+ \bigl| \left[ \tilde{L}^m \partial^j_s , \gamma^{\alpha \beta} \partial_\alpha \partial_\beta \right] w_k \bigr| \bigl] 
\lesssim\\ 
\Bigl(\ssum{|\mu| + m \leq N_0 + 3D + 31}{m \leq 1} \bigl|
L^m \partial^\mu (w_k+u_0)^\prime \bigr| + \ssum{|\mu|+m\le
  N_0+3D+33}{m\le 1}
|L^m\partial^\mu u_0|\Bigr) \sum_{|\mu| \leq N_0} \bigl| \partial^\mu (w_{k-1} + u_0) \bigr| \\
+ \ssum{|\mu| + m \leq N_0 + 3D + 31}{m \leq 1}  \bigl| L^m \partial^\mu (w_{k-1} + u_0)^\prime \bigr| \sum_{|\mu| \leq N_0 - 1} \left( \bigl| \partial^\mu (w_{k-1} + u_0) \bigr| + \bigl| \partial^\mu (w_{k} + u_0)^\prime \bigr| \right) \\
+ \ssum{|\mu| + m \leq N_0}{m \leq 1} \bigl| L^m \partial^\mu (w_k + u_0)^\prime \bigr| \sum_{|\mu| \leq N_0 + 3D + 31} \bigl( \bigl| \partial^\mu (w_{k-1} + u_0)^\prime \bigr| \bigr) \\
+ \left( \sum_{|\mu| \leq N_0 + 3D + 31} \bigl( \bigl| \partial^\mu (w_{k-1} + u_0)^\prime \bigr| + \bigl| \partial^\mu (w_{k} + u_0)^\prime \bigr| \bigr) + \sum_{|\mu| \leq N_0 + 3D + 33} \bigl| \partial^\mu u_0 \bigr| \right) \\
\times \ssum{|\mu| + m \leq N_0}{m \leq  1} \bigl| L^m \partial^\mu (w_{k-1} + u_0 ) \bigr|  \\
+ \sum_{m \leq 1} \bigl| L^m (w_{k-1} + u_0 ) \bigr| \bigl|  (w_{k-1} + u_0 ) \bigr| + \ssum{|\mu| + m \leq N_0 + 3D + 31}{m \leq 1 } \bigl| L^m \partial^\mu [\square, \eta] u \bigr| .
\end{multline}
Due to the fact that the initial data are compactly supported, it follows from finite propagation speed that
\[
\ssum{|\mu| + m \leq N_0 + 3D + 31}{m \leq 1 } \bigl| L^m \partial^\mu [\square, \eta] u \bigr| \lesssim \sum_{|\mu| \leq N_0 + 3D + 31} \bigl| \partial^\mu [\square, \eta] u \bigr|
\] 
since the coefficients of $L$ are uniformly bounded on the support of
$[\square, \eta] u$.  A similar statement holds for $u_0$.
\par
The third and fourth terms in the right hand side of \eqref{term5 ineq
  1} can be controlled by summing over dyadic intervals and applying Lemma \ref{sobolev on annuli} to the lower order terms. For more details on this computation, the reader can see \cite{KSS}. Taking the $L^2_x$-norm of these terms, we see that they are controlled by
\begin{multline}
\ssum{|\mu| + m \leq N_0 + 3}{m \leq 1} \norm{\left< x \right>^{-3/4} L^m Z^\mu (w_k + u_0)^\prime }{2} \sum_{|\mu| \leq N_0 + 3D + 31}  \norm{\left< x \right>^{-3/4} \partial^\mu (w_{k-1} + u_0)^\prime }{2} \\
+ \ssum{|\mu| + m \leq N_0 + 3}{m \leq  1} \norm{\left< x \right>^{-1/2} L^m Z^\mu (w_{k-1} + u_0 ) }{2} \times \\
\left(  \sum_{|\mu| \leq N_0 + 3D + 31} \bigl( \norm{\left< x \right>^{-3/4} \partial^\mu (w_{k-1} + u_0)^\prime}{2} + \norm{\left< x \right>^{-3/4} \partial^\mu (w_{k} + u_0)^\prime }{2} \bigr) + \sum_{|\mu| \leq N_0 + 3D + 33} \norm{\partial^\mu u_0}{2} \right).
\end{multline}
Integrating in the $t$ variable and applying Cauchy-Schwarz, we see
that these terms are controlled using \eqref{local solution bound},  $II$, and $\log ( 2 + T )^{1/2} VIII$ in \eqref{solution norm}.
The first two terms in the right hand side of \eqref{term5 ineq 1} can be controlled using terms $ V$ and $IX$.
The second-to-last term can be controlled using terms $VII$ and $IX$. 
It follows that
\begin{multline} \label{term 5 control 1}
\int^T_0 \ssum{j + m \leq N_0 + 3D + 31}{m \leq 1} \bigl[ \norm{ \tilde{L}^m \partial^j_s \square_\gamma w_k(s,\cdot)}{2}
+ \norm{ \left[ \tilde{L}^m \partial^j_s , \gamma^{\alpha \beta} \partial_\alpha \partial_\beta \right] w_k(s,\cdot)}{2}\bigr] \; ds \\
\leq C  (M_0(T) + M_{k-1}(T) + M_{k}(T))(M_0(T) + M_{k-1}(T)) \log(2 + T)^{1/2} + C_2 \epsilon.
\end{multline}
Note that $ \begin{displaystyle} \ssum{|\mu| + m \leq N_0 + 3D + 30}{m \leq 1} \bigl| L^m \partial^\mu \square w_k \bigr|  \end{displaystyle} $ is also controlled by the right hand side of \eqref{term5 ineq 1}. It follows that the last term in \eqref{term5 reduction 3} can be controlled using terms $I$ and $V$ in \eqref{solution norm} via Sobolev embedding. Arguing in this manner, we see that this term can be controlled without a logarithmic loss. The third-to-last term in \eqref{term5 reduction 3} can be bounded in the same manner as term $I$. To handle the second-to-last term in \eqref{term5 reduction 3}, we apply \eqref{boundary term ineq} to see that it is bounded by
\begin{equation} \label{term 5 ineq 2}
\int^T_0 \sum_{|\mu| \leq N_0 + 4D + 32} \norm{\partial^\mu \square w_k (s,\cdot)}{2} \; ds + \int^T_0 \int_{\ext} \sum_{|\mu| \leq N_0 + 3D + 36} \left| Z^\mu \square w_k (s,y) \right| \; \dfrac{dy \; ds}{|y|^{3/2}} .
\end{equation} 
The first term can be controlled in a manner similar to the argument that was used to bound term $I$. Applying the product rule, we see that 
\begin{multline} \label{term5 ineq 2}
\sum_{|\mu| \leq N_0 + 3D + 36} \left| Z^\mu \square w_k (s,y) \right| \lesssim  \sum_{|\mu| \leq N_0} \bigl| Z^\mu (w_{k-1} + u_0 ) \bigr|  \times \\
\bigl( \ssum{|\mu| \leq N_0 + 3D + 36}{|\nu| \leq 1} \bigl| Z^\mu \partial^\nu (w_{k-1} + u_0) \bigr| + \ssum{|\mu| \leq N_0 + 3D + 36}{|\nu| = 2} \bigl| Z^\mu \partial^\nu (w_{k} + u_0) \bigr| \bigr) \\
 + \sum_{|\mu| \leq N_0 + 3D + 36} \bigl| \partial^\mu [\square, \eta] u \bigr| .
\end{multline}
The last term can be controlled using \eqref{local solution bound}. Thus, by \eqref{term5 ineq 2} and Cauchy-Schwarz, we see that the second term in \eqref{term 5 ineq 2} is controlled by
\begin{multline} \label{term 5 ineq 3} 
\sum_{|\mu| \leq N_0} \norm{ \left< x \right>^{-1/2} Z^\mu (w_{k-1} + u_0 )}{L^2_{t,x}(S_T)}  \Bigl( \ssum{|\mu| \leq N_0 + 3D + 36}{|\nu| \leq 1} \norm{ \left< x \right>^{-1/2} Z^\mu \partial^\nu (w_{k-1} + u_0) }{L^2_{t,x} (S_T)} \\
+ \ssum{|\mu| \leq N_0 + 3D + 36}{|\nu| = 2} \norm{ \left< x \right>^{-1/2} Z^\mu \partial^\nu (w_{k} + u_0) }{L^2_{t,x}(S_T)} \Bigr) + \epsilon.
\end{multline}
Since each factor in \eqref{term 5 ineq 3} is controlled by $\log ( 2 + T)^{1/2} IV$ in \eqref{solution norm}, we see that
\begin{multline} \label{term 5 ineq 4}
 \int^T_0 \int_{\ext}  \sum_{|\mu| \leq N_0 + 3D + 36} \left| Z^\mu \square w_k (s,y) \right| \; \dfrac{dy \; ds}{|y|^{3/2}} \\
 \leq C (M_0(T) + M_{k-1}(T) + M_k(T))(M_0(T) + M_{k-1}(T))\log(2+T) + C_2 \epsilon.
\end{multline}
From this and previous arguments, we conclude that
\[
V \leq C (M_0(T) + M_{k-1}(T) + M_k(T))(M_0(T) + M_{k-1}(T))\log(2+T) + C_2 \epsilon.
\]
\par
{\em{\bf{Bound for $VI$: }}}
Controlling terms $VI, VII$ and $VIII$ will be conducted in a manner
similar to the analogous terms where no $L$ is being applied. In the
case of term $VI$, we apply a cutoff that is supported away from the
obstacle $\K$ to $L \partial^\mu w_k$ and apply \eqref{main energy
  ineq}. For the remaining term that is supported near $\K$ and the
commutator term that results from the cutoff we apply 
\eqref{localized energy local control}. Thus, it suffices for us to control
\begin{equation} \label{term 6 reduction}
\int^T_0 \ssum{|\mu| + m \leq N_0 + 3D + 21}{m \leq 1} \norm{L^m \partial^\mu \square w_k (s,\cdot)}{2} \; ds + \ssum{|\mu| + m \leq N_0 + 2D + 20}{m \leq 1} \norm{L^m \partial^\mu \square w_k}{L^2_{t,x}(S_T )} .
\end{equation}
One only needs to demonstrate control for the first term in the above expression since the second term can be controlled in a simpler manner by applying Sobolev embedding to the lower order terms. Applying the product rule, we see that
\begin{multline} \label{term 6 ineq1}
\ssum{|\mu| + m \leq N_0 + 3D + 21}{m \leq 1} \bigl| L^m \partial^\mu \square w_k \bigr| \lesssim \ssum{|\mu| + m \leq N_0 + 3D + 22}{m \leq 1} \bigl| L^m \partial^\mu (w_k + u_0)^\prime \bigr| \sum_{|\mu| \leq N_0} \bigl| \partial^\mu (w_{k-1} + u_0) \bigr| \\
 + \ssum{|\mu| + m \leq N_0 + 3D + 21}{m \leq 1} \bigl| L^m \partial^\mu (w_{k-1} + u_0)^\prime \bigr| \sum_{|\mu| \leq N_0 - 1} \bigl( \bigl| \partial^\mu (w_{k} + u_0)^\prime \bigr| + \bigl| \partial^\mu (w_{k-1} + u_0) \bigr| \bigr) \\
+ \sum_{|\mu| \leq N_0 + 3D + 22} \bigl| \partial^\mu (w_k + u_0)^\prime \bigr| \ssum{|\mu| + m \leq N_0}{m \leq 1} \bigl| L^m \partial^\mu (w_{k-1} + u_0) \bigr| \\
 + \sum_{|\mu| \leq N_0 + 3D + 21} \bigl| \partial^\mu (w_{k-1} + u_0)^\prime \bigr| \ssum{|\mu| + m\leq N_0}{m \leq 1} \bigl( \bigl| L^m \partial^\mu (w_{k} + u_0)^\prime \bigr| + \bigl| L^m \partial^\mu (w_{k-1} + u_0) \bigr| \bigr) \\
 + \bigl| w_{k-1} + u_0 \bigr| \sum_{m \leq 1} \bigl| L^m (w_{k-1} + u_0 ) \bigr| + \sum_{|\mu| \leq N_0 + 3D + 21} \bigl| \partial^\mu [\square, \eta] u \bigr| .
\end{multline}
Upon taking the $L^1([0,T];L^2(\ext))$-norm, we bound the
corresponding terms in a manner that is reminiscent of the above arguments.
Where $L$ is being applied to the higher order terms, we use terms $V$ and $IX$ of \eqref{solution norm}. The second-to-last term in \eqref{term 6 ineq1} can be controlled using terms $VII$ and $IX$. The last term in \eqref{term 6 ineq1} can be controlled using \eqref{local existence}. The remaining terms can be controlled by decomposing dyadically in the $x$ variable, applying \eqref{sobolev on annuli}, and applying Cauchy-Schwarz in $t$ and in the dyadic summation variable. This results in quantities that can be controlled by $\log ( 2 + T)^{1/2} IV$ and $\log ( 2 + T)^{1/2} VIII$.
Therefore, we see that
\[
VI \leq C (M_0(T) + M_k(T) + M_{k-1}(T)) (M_0(T) + M_{k-1}(T)) \log (2 + T) + C_2 \epsilon .
\]
\par
{\em{\bf{Bound for $VII$ and $VIII$, $|\nu| = 0$: }}}
We closely follow the preceding argument that was used to control
$III$ and $IV$, $ |\nu| = 0$. Fix $\chi$ just as before. The Dirichlet
boundary conditions and the boundedness of the coefficients of $Z$ on
the boundary of $\K$ demonstrate that $ \chi L Z^\mu w_k$ 
is bounded by terms $V$ and $VI$.
It, thus, suffices to prove bounds for $(1 - \chi) L Z^\mu w_k$ for $|\mu| \leq N_0 + D + 10$ fixed. We write
\begin{multline}
\square (1 - \chi) L Z^\mu w_k = (1 - \eta)(1 - \chi) \partial_\alpha \bigl[ (1 - \eta ) (1 - \chi) \bigl( b^{\alpha \beta} (u_0 + w_{k-1}) L Z^\mu \partial_\beta (u_0 + w_{k})  \\
 + b^{\alpha \beta \gamma} \partial_\gamma (u_0 + w_{k-1}) L Z^\mu \partial_\beta (u_0 + w_{k}) \bigr) \bigr] - [\square, \chi] L Z^\mu w_k + \tilde{G}_{k,\mu}(t,x) .
\end{multline}
Applying \eqref{divergence form}, \eqref{cutoff estimate} and \eqref{u bound}, we see that
\begin{multline} \label{terms7and8 ineq1}
\sup_{0 \leq t \leq T} \norm{ (1 - \chi) L Z^\mu w_k (t,\cdot)}{2} + \log(2 + T)^{-1/2} \norm{( 1 - \chi) \left< x \right>^{-1/2} L Z^\mu w_k }{L^2_{t,x}(S_T)} \\
\lesssim \int^T_0 \norm{\left< x \right>^{-1} \tilde{G}_{k,\mu} (s,\cdot) }{L^1_r L^2_\omega(|x| > 2)} \; ds \\
+ \int^T_0 \sum_{|\theta| \leq N_0 + D + 10} \sum_{|\lambda| \leq 1} \norm{|L Z^\theta (w_k + u_0 )^\prime| |\partial^\lambda (w_{k-1} + u_0)|}{2} \; ds \\
+ \ssum{|\theta| + m \leq N_0 + D + 11}{m \leq 1} \norm{L \partial^\theta w_k^\prime}{L^2_{t,x}(S_T \cap \{ |x| < 4 \})} .
\end{multline}
The last term is controlled by term $VI$. Using terms $VII$ and $IX$, the second-to-last term is controlled by
\[
(M_0(T) + M_{k-1}(T))(M_0(T) + M_k(T)) \log( 2 + T ) .
\]
The first term in \eqref{terms7and8 ineq1} can be controlled using a rough application of the product rule where we allow for at most one occurrence of the scaling vector field on each term. Applying Sobolev embedding on $S^3$ and Cauchy-Schwarz, we see that
\begin{multline*} 
\int^T_0 \norm{\left< x \right>^{-1} \tilde{G}_{k,\mu} (s,\cdot) }{L^1_r L^2_\omega(|x| > 2)} \; ds \lesssim \epsilon \\
+ \ssum{|\theta| + m \leq N_0 + D + 11}{m \leq 1 }\norm{\left< x \right>^{-1/2} L^m Z^\theta  (w_k + u_0)^\prime}{L^2_{t,x}(S_T)} \ssum{|\theta| + m \leq N_0 + 2}{m \leq 1 }\norm{\left< x \right>^{-1/2} L^m Z^\theta  (w_{k-1} + u_0)}{L^2_{t,x}(S_T)} \\
+ \ssum{|\theta| + m \leq N_0 + D + 11}{m \leq 1 \\ |\lambda| \leq 1}\norm{\left< x \right>^{-1/2} L^m Z^\theta \partial^\lambda (w_{k-1} + u_0)}{L^2_{t,x}(S_T)} \ssum{|\theta| + m \leq N_0 + 2}{m \leq 1 }\norm{\left< x \right>^{-1/2} L^m Z^\theta  (w_{k-1} + u_0)}{L^2_{t,x}(S_T)} \\
+ \ssum{|\theta| + m \leq N_0 + 2}{m \leq 1 }\norm{\left< x \right>^{-1/2} L^m Z^\theta (w_k + u_0)^\prime}{L^2_{t,x}(S_T)} \ssum{|\theta| + m \leq N_0 + D + 11}{m \leq 1 \\ |\lambda| \leq 1}\norm{\left< x \right>^{-1/2} L^m Z^\theta \partial^\lambda (w_{k-1} + u_0)}{L^2_{t,x}(S_T)} .
\end{multline*}
Using terms $IV$ and $VIII$, it follows that
\begin{multline} \label{terms7and8 remainder term}
\int^T_0 \norm{\left< x \right>^{-1} \tilde{G}_{k,\mu} (s,\cdot) }{L^1_r L^2_\omega(|x| > 2)} \; ds \\
\lesssim \epsilon + (M_0(T) + M_{k-1}(T) + M_k(T))(M_0(T) + M_{k-1}(T)) \log(2 + T),
\end{multline}
and therefore
\[
VII \bigr|_{|\nu| = 0}  + VIII \bigr|_{|\nu| = 0} \leq C  (M_0(T) + M_{k-1}(T) + M_k(T))(M_0(T) + M_{k-1}(T)) \log(2 + T) + C_2 \epsilon .
\]
\par
{\em{\bf{Bound for $VII$ and $VIII$, $|\nu| = 1$: }}}
Again, we need only provide bounds for $( 1 - \chi ) L Z^\mu
w_k^\prime$ since $\chi L Z^\mu w_k^\prime$ can be controlled using
terms $V$ and $VI$. Applying \eqref{energy} and \eqref{main energy ineq}, we see that
\begin{multline}
\sup_{0 \leq t \leq T} \sum_{|\mu| \leq N_0 + D + 10} \norm{(1 - \chi) L Z^\mu w_k^\prime (t,\cdot)}{2} + \sum_{|\mu| \leq N_0 + D + 10} \norm{(1 - \chi) \left< x \right>^{-3/4} L Z^\mu w_k^\prime}{L^2_{t,x}(S_T)} \\
 \lesssim \int^T_0 \ssum{|\mu| + m \leq N_0 + D + 11}{m \leq 1} \norm{L^m Z^\mu \square w_k(s,\cdot)}{2} \; ds + \ssum{|\mu| + m \leq N_0 + D + 11}{m \leq 1} \norm{L^m \partial^\mu w_k^\prime}{L^2_{t,x}([0,T] \times \{|x| < 4 \})} .
\end{multline}
The second term in the right hand side is controlled by term $VI$. The first term can be handled by carefully applying the product rule. We see that
\begin{multline} \label{terms7and8 ineq2}
\ssum{|\mu| + m \leq N_0 + D + 11}{m \leq 1} \bigl| L^m Z^\mu \square w_k \bigr| \lesssim  \ssum{|\mu| + m \leq N_0 + D + 11}{m \leq 1 \\ |\theta| = 2} \bigl| L^m Z^\mu \partial^\theta (w_k + u_0) \bigr|  \sum_{|\mu| \leq N_0} \bigl| Z^\mu (w_{k-1} + u_0) \bigr| \\
+ \ssum{|\mu| \leq N_0 + D + 11}{|\theta| = 2} \bigl| Z^\mu \partial^\theta (w_k + u_0) \bigr| \ssum{m + |\mu| \leq N_0}{m \leq 1} \bigl| L^m Z^\mu (w_{k-1} + u_0) \bigr| \\
+ \ssum{|\mu| + m \leq N_0 + D + 11}{m \leq 1 \\ |\theta| \leq 1} \bigl| L^m Z^\mu \partial^\theta (w_{k-1} + u_0) \bigr| \ssum{|\mu| + m \leq N_0}{m \leq 1} \left( \bigl| L^m Z^\mu (w_{k-1} + u_0) \bigr| + \bigl| L^m Z^\mu (w_k + u_0)^\prime \bigr| \right) \\
+ \ssum{|\mu| + m  \leq N_0 + D + 11}{m \leq 1} \bigl| L^m Z^\mu [\square, \eta] u \bigr|.
\end{multline}
By finite propagation speed and the fact that the initial data are compactly supported the coefficients of $L$ and $Z$ are bounded uniformly on the support of $[\square, \eta] u$. Thus, the last term in \eqref{terms7and8 ineq2} is handled using \eqref{local existence}.
Taking the $L^1([0,T] ; L^2(\ext))$-norm of the first term of the
right hand side, we see that it is controlled by terms $III$, $VII$, and $IX$:
\begin{multline}
\int^T_0 \ssum{|\mu| + m \leq N_0 + D + 11}{m \leq 1 \\ |\theta| = 2} \norm{ L^m Z^\mu \partial^\theta (w_k + u_0) (s,\cdot)}{2} \sum_{|\mu| \leq N_0} \norm{ Z^\mu (w_{k-1} + u_0) (s,\cdot)}{\infty} \; ds \\
\leq C (M_k(T) + M_0(T))(M_{k-1}(T) + M_0(T)) \log(2 + T) .
\end{multline}
To handle the remaining terms, we decompose dyadically and apply \eqref{sobolev on annuli} to the lower order pieces. We see that they can be controlled by
\begin{multline}
\ssum{|\mu| \leq N_0 + D + 11}{|\theta| = 2} \norm{ \left< x \right>^{-3/4} Z^\mu \partial^\theta (w_k + u_0) }{L^2_{t,x}(S_T)} \ssum{m + |\mu| \leq N_0 + 3}{m \leq 1} \norm{ \left< x \right>^{-1/2} L^m Z^\mu (w_{k-1} + u_0) }{L^2_{t,x}(S_T)} \\
+ \ssum{|\mu| + m \leq N_0 + D + 11}{m \leq 1 \\ |\theta| \leq 1} \norm{ \left< x \right>^{-1/2} L^m Z^\mu \partial^\theta (w_{k-1} + u_0) }{L^2_{t,x}(S_T)} \ssum{|\mu| + m \leq N_0 + 3}{m \leq 1} \norm{ \left< x \right>^{-1/2} L^m Z^\mu (w_{k-1} + u_0) }{L^2_{t,x}(S_T)} \\
+ \ssum{|\mu| + m \leq N_0 + 3}{m \leq 1} \norm{ \left< x \right>^{-3/4} L^m Z^\mu (w_k + u_0)^\prime }{L^2_{t,x}(S_T)} \ssum{|\mu| + m \leq N_0 + D + 11}{m \leq 1 \\ |\theta| \leq 1} \norm{ 
\left< x \right>^{-1/2} L^m Z^\mu \partial^\theta (w_{k-1} + u_0) }{L^2_{t,x}(S_T)} .
\end{multline}
Since all the factors in the above expression are controlled by either $\log ( 2 + T ) ^{1/2} IV$ or $\log ( 2 + T )^{1/2} VIII$, these terms are bounded by
\[
C (M_{k-1} + M_k(T) + M_0(T))(M_{k-1}  + M_0(T)) \log( 2 + T).
\]
Therefore, we see that
\[
VII \bigr|_{|\nu| = 1} + VIII\bigr|_{|\nu| = 1} \leq C (M_{k-1} + M_k(T) + M_0(T))(M_{k-1}  + M_0(T)) \log( 2 + T) + C_2 \epsilon .
\]
\par
{\em{\bf{Bound for $VII$, $|\nu| = 2$: }}}
We let $\gamma$ be as before in \eqref{term1 perturbation
  terms}. Letting $\delta = C_1 \epsilon$, we apply \eqref{perturbed
  energy full vector fields}. Integrating the resulting inequality
over the timestrip $[0,T]$ and applying \eqref{gamma prime small} and \eqref{lifespan}, we see that $VII \bigr|_{|\nu| = 2}$ is controlled by
\begin{multline} \label{terms7and8 ineq3}
\int^T_0 \ssum{|\mu| + m \leq N_0 + D + 11 }{m \leq 1 \\ |\lambda| = 1} \norm{\square_\gamma L^m Z^\mu \partial^\lambda w_k(s,\cdot)}{2} \; ds + \ssum{|\mu| + m \leq N_0 + D + 13}{m \leq 1} \norm{L^m \partial^\mu \square w_k}{L^2_{t,x} (S_T \cap \{ |x| < 2 \})} .
\end{multline}
The second term in \eqref{terms7and8 ineq3} is controlled using an argument similar to the one used to bound term $VI$. To control the integrand inside the first term, we apply the product rule to see that
\begin{multline} \label{terms7and8 ineq4}
\ssum{|\mu| + m \leq N_0 + D + 11}{m \leq 1 \\ |\lambda| = 1} \bigl| \square_\gamma L^m Z^\mu \partial^\lambda w_k \bigr| \\
\lesssim \bigl( \ssum{|\mu| + m \leq N_0 + D + 11}{m \leq 1 \\ |\lambda| = 2} \bigl| L^m Z^\mu \partial^\lambda (w_k + u_0 ) \bigr| + \sum_{|\mu| \leq N_0 + D + 14} \bigl| \partial^\mu u_0 \bigr| \bigr) \sum_{|\mu| \leq N_0} \bigl| Z^\mu (w_{k-1} + u_0 ) \bigr| \\
+ \ssum{|\mu| + m \leq N_0 + D + 11}{m \leq 1 \\ |\lambda| \leq 2} \bigl| L^m Z^\mu \partial^\lambda (w_{k-1} + u_0 ) \bigr| \sum_{|\mu| \leq N_0 - 1} \bigl( \bigl| Z^\mu (w_k + u_0 )^\prime \bigr| + \bigl| Z^\mu (w_{k-1} + u_0 ) \bigr| \bigr) \\
+ \bigl( \ssum{|\mu| \leq N_0 + D + 11}{ |\lambda| = 2} \bigl| Z^\mu \partial^\lambda (w_k + u_0 ) \bigr| + \sum_{|\mu| \leq N_0 + D + 14} \bigl| \partial^\mu u_0 \bigr| \bigr) \ssum{|\mu| + m \leq N_0}{m \leq 1} \bigl| L^m Z^\mu (w_{k-1} + u_0 ) \bigr| \\
+ \ssum{|\mu| \leq N_0 + D + 11}{|\lambda| \leq 2} \bigl| Z^\mu \partial^\lambda (w_{k-1} + u_0 ) \bigr|  \ssum{|\mu| + m \leq N_0}{m \leq 1} \left( \bigl| L^m Z^\mu (w_k + u_0 )^\prime \bigr| +  \bigl| L^m Z^\mu (w_{k-1} + u_0 ) \bigr| \right) \\
+ \ssum{|\mu| + m \leq N_0 + D + 12}{m \leq 1} \bigl| L^m Z^\mu [\square, \eta] u \bigr| .
\end{multline}
By finite propagation speed and the fact that the Cauchy data are compactly supported, the coefficients of $L$ and $Z$ are bounded on the support of $[\square, \eta]u$. 
\par
The $L^1([0,T]; L^2(\ext))$-norm of the second term in \eqref{terms7and8 ineq4} is controlled by
\begin{multline}
\int^T_0  \ssum{|\mu| + m \leq N_0 + D + 11}{m \leq 1 \\ |\lambda| = 2} \norm{ L^m Z^\mu \partial^\lambda (w_{k-1} + u_0 ) (s,\cdot)}{2} \times \\
\sum_{|\mu| \leq N_0} \bigl(  \norm{ Z^\mu (w_{k} + u_0 )(s,\cdot) }{\infty} + \norm{ Z^\mu (w_{k-1} + u_0 )(s,\cdot) }{\infty} \bigr) \; ds, 
\end{multline}
which can be handled using $III$, $VII$ and $IX$. The first term in the right hand side of \eqref{terms7and8 ineq4} can be handled in an analogous manner. The third and fourth terms can be handled using a dyadic decomposition in the $x$ variable followed by an application of Lemma \ref{sobolev on annuli}. We will illustrate this method by bounding the fourth term. Taking the $L^2_x$-norm of this term, we see that they can be controlled by
\begin{multline}
 \ssum{|\mu| \leq N_0 + D + 11}{|\lambda| \leq 2} \norm{ \left< x \right>^{-1/2} Z^\mu \partial^\lambda (w_{k-1} + u_0 ) }{2} \times \\
\Bigl( \ssum{|\mu| + m \leq N_0 + 3}{m \leq 1} \norm{ \left< x \right>^{-3/4} L^m Z^\mu (w_k + u_0 )^\prime }{2} + \ssum{|\mu| + m \leq N_0 + 3}{m \leq 1} \norm{ \left< x \right>^{-1/2} L^m Z^\mu (w_{k-1} + u_0 ) }{2} \Bigr) .
\end{multline}
Integrating in time and applying Cauchy-Schwarz, we see that this is controlled by $\log ( 2 + T)^{1/2} IV$ and $\log ( 2 + T)^{1/2} VIII$. The third term in the right hand side of \eqref{terms7and8 ineq4} can be controlled in a similar manner. From these arguments, we see that
\begin{multline*}
\int^T_0 \ssum{|\mu| + m \leq N_0 + D + 11 }{m \leq 1 \\ |\lambda| = 1} \norm{\square_\gamma L^m Z^\mu \partial^\lambda w_k(s,\cdot)}{2} \; ds \\
\leq C (M_k(T) + M_{k-1}(T) + M_0(T))( M_{k-1}(T) + M_0(T))\log(2+T) + C_2 \epsilon .
\end{multline*}
Therefore, we conclude that
\[
VII \bigr|_{|\nu| = 2} \leq C (M_k(T) + M_{k-1}(T) + M_0(T))( M_{k-1}(T) + M_0(T)) \log ( 2 + T ) + C_2 \epsilon .
\]
\par
{\em{\bf{Bound for $IX$: }}}
Applying \eqref{klainerman sideris no gradient ext domain} and \eqref{boundary term ineq}, we see that
\begin{multline} \label{term9 main ineq}
IX \lesssim 
\sup_{t \in [0,T]} \left< t \right> \sum_{|\mu| \leq N_0 + 3} \norm{|x|^{-1} Z^\mu \square w_k (t,\cdot)}{L^1_r L^2_\omega (|x| > 2)} \\ 
 + \int^T_0 \ssum{|\mu| + m \leq N_0 + D + 4}{m \leq 1} \norm{ L^m \partial^\mu  \square w_k (s,\cdot)}{2} \; ds \\
 + \int^T_0 \ssum{|\mu| + m \leq N_0 + 3}{m \leq  1} \norm{|x|^{-1} L^m Z^\mu \square w_k (s,\cdot)}{L^1_r L^2_\omega (|x| > 2)} \; ds \\
 + \int^T_0 \int_{\ext} \sum_{|\mu| \leq N_0 +D+ 8} \left| Z^\mu \square w_k (s,y) \right| \; \dfrac{dy \; ds}{|y|^{3/2}} .
\end{multline}
We see that the integrand in the second term in the right hand side of
\eqref{term9 main ineq} is controlled by the right hand side of
\eqref{term 6 reduction}.  Thus, it follows from the same argument
used to control \eqref{term 6 reduction} that this term is bounded by
$C (M_k(T) + M_{k-1}(T) + M_0(T))(M_{k-1}(T) + M_0(T)) \log (2 + T) +
C_2 \epsilon$. We also see that the third and fourth terms are
controlled by the right hand sides of \eqref{terms7and8 remainder
  term} and \eqref{term 5 ineq 4}, respectively. Therefore, it
suffices to demonstrate how to control the first term in the right hand side of \eqref{term9 main ineq}.

To deal with the first term, we see that it is controlled by
\begin{multline} \label{second term split}
\sup_{t \in [0,T]} \left< t \right> \sum_{|\mu| \leq N_0 + 3} \int_2^{\max(2,t/2)} \left( \int_{S^3} \left| Z^\mu \square w_k (t,r \omega) \right|^2 \; d\omega \right)^{1/2} \; r^2 \; dr \\ 
+ \sup_{t \in [0,T]} \sum_{|\mu| \leq N_0 + 3} \int_{\max(2,t/2)}^\infty \left( \int_{S^3} \dfrac{\left< t \right>}{r} \left| Z^\mu \square w_k (t,r \omega) \right|^2 \; d\omega \right)^{1/2} \; r^3 \; dr .
\end{multline}
Due to the fact that $t \lesssim r$ and $r > 2$ on the region of integration for the second integral in \eqref{second term split}, we see that this term is controlled by
\begin{equation} \label{second term2}
\sup_{t \in [0,T]} \sum_{|\mu| \leq N_0 + 3} \norm{ Z^\mu \square w_k (t, \cdot )}{L^1_r L^2_\omega (|x| > 2)}  .
\end{equation}
Applying Cauchy-Schwarz and Sobolev embedding on $S^3$, the above quantity can be controlled using term $III$ and \eqref{local existence}. Thus, we see that \eqref{second term2} is bounded by
\[
C (M_k(T) + M_{k-1}(T) + M_0 (T)) (M_{k-1}(T) + M_0(T)) + C_2 \epsilon .
\]
To control the first integral in \eqref{second term split}, we apply Cauchy-Schwarz in the $r$ variable to see that
\begin{multline} \label{first integral ineq}
\sup_{t \in [0,T]} \left< t \right> \sum_{|\mu| \leq N_0 + 3} \int_2^{\max(2,t/2)} \left( \int_{S^3} \left| Z^\mu \square w_k (s,r \omega) \right|^2 \; d\omega \right)^{1/2} \; r^2 \; dr \\
\lesssim \log ( 2 + T )^{1/2} \sup_{t \in [0,T]} \left< t \right> \sum_{|\mu| \leq N_0 + 3} \norm{r Z^\mu \square w_k(t,\cdot)}{L^2_x( \{ x \in \ext : 2 < |x| < \max(2, t/2) \} )} .
\end{multline}
Note that for $v_1,v_2 \in C^\infty (\R \times \ext ),$
\[
\left< t \right> \norm{r v_1(t,\cdot) v_2(t,\cdot)}{L^2_x( \{ x \in \ext : 2 < |x| < \max(2, t/2) \} )} \lesssim \norm{\left< r \right> \left< t - r \right> v_1 (t,\cdot)}{\infty} \norm{v_2(t,\cdot)}{2} .
\] 
From this observation and the fact that
\begin{multline}
	\sum_{|\mu| \leq N_0 + 3} | Z^\mu \square w_k | 
	\lesssim \sum_{|\mu| \leq N_0} \bigl| Z^\mu ( w_{k-1} + u_0 ) \bigr|  \ssum{|\mu| \leq N_0 + 3}{|\nu | = 2} \bigl| Z^\mu \partial^\nu (w_k + u_0) \bigr|  \\
	+ \sum_{|\mu| \leq N_0 - 1} \bigl| Z^\mu ( u_0 + w_{k} )^\prime \bigr| \sum_{|\mu| \leq N_0 + 3} \bigl| Z^\mu (  w_{k-1} + u_0 )^\prime \bigr| \\
	+ \sum_{|\mu| \leq N_0} \bigl| Z^\mu ( u_0 + w_{k-1} ) \bigr| \sum_{|\mu| \leq N_0 + 4} \bigl| Z^\mu (  w_{k-1} + u_0 ) \bigr| \\
+ \sum_{|\mu| \leq N_0 + 3} \bigl| Z^\mu [ \square , \eta ] u \bigr| ,
\end{multline}
we use terms $III$ and $IX$ in \eqref{solution norm} to see that the right hand side of \eqref{first integral ineq} is controlled by
\begin{equation*}
 (M_0(T) + M_k(T) + M_{k-1}(T))(M_0(T) + M_{k-1}(T))\log(2+T)^{1/2} + \epsilon.
\end{equation*}
Hence, it follows that
\[
IX \leq C (M_k(T) + M_{k-1}(T) + M_0 (T)) (M_{k-1}(T) + M_0(T)) \log(2+T) + C_2 \epsilon .
\]
\par
{\em{\bf{Boundness of $M_k(T)$: }}}  Here we show that
\eqref{induction step} implies \eqref{induction hypothesis} with the
same uniform constant $C_1$. 
Combining the estimates that we have obtained for $I, \ldots, IX$, we see that
\[
M_k(T) \leq C (M_k(T) + M_{k-1} + M_0(T))(M_{k-1} + M_0(T)) \log(2 + T) + C_2 \epsilon.
\]
If we pick $C_1$ such that $C_1 > 2 C_2$, then applying \eqref{base case} and \eqref{induction hypothesis} yields the inequality
\[
M_k(T) \leq C (M_k(T) + \epsilon)\epsilon \log(2 + T) + \dfrac{C_1}{2} \epsilon .
\]
If $\epsilon, c$ in \eqref{data} and \eqref{lifespan} are sufficiently small, then \eqref{induction step} follows.
\par
{\em{\bf{Convergence of $\{ w_k \}$: }}}
We shall now show that uniform boundedness of each $M_k(T)$ implies that the sequence $\{ w_k \}$ is Cauchy. Standard results show that this implies that $\{ w_k \}$ converges to a solution to \eqref{zero data}, which implies Theorem \ref{thm1}. Setting
\begin{multline*}
A_k(T) = \sup_{t \in [0,T]}\sum_{|\mu| \leq D + 20} \norm{\partial^\mu(w_{k} - w_{k-1})(t,\cdot)}{2} \\
+ \sum_{|\mu| \leq 2} \log(2 + T)^{-1/2} \norm{\left< x \right>^{-1/2} \partial^\mu (w_k - w_{k-1})}{L^2_{t,x}(S_T)},
\end{multline*}
similar arguments to those used to bound $I,\ldots, IX$ along with \eqref{induction step} imply that
\[
A_k(T) \leq \dfrac{1}{2} A_{k-1}(T)
\]
for $T \leq T_\epsilon$, provided that $\epsilon, c$ are sufficiently small. This immediately yields that $\{ w_k \}$ is Cauchy in the space $X_T$ whose norm is given by
\[
\norm{v}{X_T} = \sup_{t \in [0,T]}\sum_{|\mu| \leq D + 20} \norm{\partial^\mu v(t,\cdot)}{2} .
\]
This completes the proof.

\end{document}